\documentclass[11pt]{amsart}
\usepackage{amssymb,latexsym, amscd}
\usepackage[square, numbers]{natbib}
\usepackage[all]{xy}
\usepackage{graphicx}
\usepackage{mathrsfs}
\usepackage{amsmath}

 \newtheorem{theorem}{Theorem}[section]
 \newtheorem{cor}[theorem]{Corollary}

 \newtheorem{lemma}[theorem]{Lemma}
 \newtheorem{proposition}[theorem]{Proposition} \theoremstyle{definition}
 \newtheorem{definition}[theorem]{Definition}
 \theoremstyle{definition}
 
 \theoremstyle{remark}
 \newtheorem{rem}[theorem]{Remark}
 \numberwithin{equation}{section}

\newcommand{\ben}{\begin{equation}}
\newcommand{\een}{\end{equation}}

\newcounter{commentcounter}

\newcommand{\integer}{\ensuremath{{\mathbb Z}}}

\newcommand{\complex}{\ensuremath{{\mathbb C}}}

\newcommand{\DD}{{\mathcal D}}

\newcommand{\XX}{{\mathcal X}}
\newcommand{\BB}{{\mathcal B}}

\newcommand{\QQ}{{\mathcal Q}}
\newcommand{\UU}{{\mathcal U}}

\newcommand{\WW}{{\mathcal W}}

\newcommand{\FF}{{\mathcal F}}
\newcommand{\EE}{{\mathcal E}}

\newcommand{\CC}{\mathcal{C}}

\newcommand{\MM}{\mathcal{M}}
\newcommand{\HH}{\mathcal{H}}

\newcommand{\OO}{\mathcal{O}}
\newcommand{\YY}{\mathcal{Y}}

\newcommand{\Tor}{\mathrm{Tor}}

\newcommand{\target}{\mathsf{target}}
\newcommand{\source}{\mathsf{source}}
\newcommand{\identity}{\mathsf{identity}}
\newcommand{\inverse}{\mathsf{inverse}}
\newcommand{\composition}{\mathsf{comp}}

\newcommand{\To}{\longrightarrow}

\newcommand{\Bun}{\ensuremath{{\mathrm{Bun}}}}
\newcommand{\Fred}{\ensuremath{{\mathrm{Fred}}}}

\begin{document}

\title[Universal twist in Equivariant K-theory]
{Universal twist in Equivariant K-theory for proper and discrete actions}

\author[N. B\'arcenas, J. Espinoza, M. Joachim and B. Uribe]{No\'e B\'arcenas, Jes\'us Espinoza, Michael Joachim and Bernardo Uribe }
\thanks{The first author was partially supported by the grant SFB 878 ``Groups, geometry  and
actions" and the Hausdorff Center for Mathematics. The first and
fourth author were partially supported with funds from the Leibniz
prize of Prof. Dr. Wolfgang L\"uck. Part of this work was done while the fourth author
was financially supported by the Alexander Von Humboldt Foundation.}

\address{Mathematisches Institut, Universit\"at Bonn, Endenicher Allee 60, D-53115 Bonn, GERMANY}
\email{barcenas@math.uni-bonn.de}
\address{Departamento de matem\'aticas, Universidad de Sonora,
 Blvd Lu\'is Encinas y Rosales S/N, Colonia Centro. Edificio 3K-1. C.P. 83000. Hermosillo, Sonora, M\'EXICO }
 \email{goku.espinoza@gmail.com}
 \address{Mathematisches Institut, Westf\"alische Wilhelms-Universit\"at,
Einsteinstrasse 62, 48149 M\"unster, GERMANY}
 \email{joachim@math.uni-muenster.de}
\address{ Departamento de Matem\'{a}ticas, Universidad de los Andes,
Carrera 1 N. 18A - 10, Bogot\'a, COLOMBIA}
 \email{ buribe@uniandes.edu.co}



\begin{abstract}
We define equivariant projective unitary stable bundles as 
the appropriate twists when defining K-theory as sections of bundles with fibers
the space of Fredholm operators over a Hilbert space.
We construct universal equivariant projective unitary stable bundles for the orbit types, and we use 
a specific model for these local universal spaces in order to glue them 
to obtain a universal  equivariant projective unitary stable bundle for discrete and proper actions.
 We determine
the homotopy type of the universal  equivariant projective unitary stable bundle, and we show that 
 the isomorphism classes of equivariant projective unitary
stable bundles are classified by the third
equivariant integral cohomology group. The results contained in
this paper extend and generalize results of Atiyah-Segal.
\end{abstract}

\maketitle
\section*{Introduction}

Topological K-theory is a generalized cohomology theory
\cite{Atiyah-book} that in the case of compact spaces can be
represented by isomorphism classes of vector bundles. A remarkable
theorem of Atiyah \cite{Atiyah-Fredholm} and J\"anich
\cite{Janich} tells us that $K^0(X) \cong
\pi_0(Maps(X,\Fred(\HH)_{\rm{norm}})$, namely that the K-theory
groups of a compact space $X$ can be alternatively obtained as the
homotopy classes of maps from the space $X$ to the space
$\Fred(\HH)$ of Fredholm operators on a fixed separable Hilbert
space $\HH$, whenever $\Fred(\HH)$ is endowed with the norm
topology. Note that the space of maps from $X$ to $\Fred(\HH)$ can
also be defined as the space of sections of the trivial bundle
$\Fred(\HH) \times X \to X$; this simple remark leads the way to
consider spaces of sections of non trivial bundles over $X$ with
fibers $\Fred(\HH)$, and by doing so we reach one of the
definitions of the twisted K-theory groups \cite{AtiyahSegal}. Let us see how this works:
the structural group of a bundle with
fiber $\Fred(\HH)$ will be the group $\UU(\HH)$ of unitary
operators on the Hilbert space endowed with the norm topology acting
 on $\Fred(\HH)$ by conjugation. As the conjugation by
complex numbers of norm one is trivial, the action of $\UU(\HH)$
factors through the group of projective unitary operators
$P\UU(\HH):=\UU(\HH)/S^1$. Therefore any principal
$P\UU(\HH)$-bundle, or in other words, any projective unitary
bundle $P\UU(\HH) \to P \to X$, provides the essential information
in order to define a bundle over $X$ with fibers $\Fred(\HH)$ by
taking the associated bundle $\Fred(P):= P \times_{P\UU(\HH)}
\Fred(\HH)$. With these bundles at hand, the twisted K-theory
groups of $X$ twisted by $P$ are defined as
$$K^{-i}(X;P) := \pi_i(\Gamma(X, \Fred(P)))$$
where $\Gamma(X, \Fred(P))$ denotes the space of sections of the
bundle $\Fred(P)$.

The equivariant version of the construction presented above turns
out to be very subtle. Not only there are issues with the topology
of the spaces $\Fred(\HH)$ and $\UU(\HH)$, as it was noted  and
resolved in \cite{AtiyahSegal}, but moreover there is not a
Hilbert space endowed with a group action that will serve as a
universal equivariant Hilbert space for projective
representations. This last fact makes the classification of twists
for equivariant K-theory a more elaborate task, that we have
pursued in this paper in a systematic way, and whose results are
the main point of this publication.

The paper is divided in three chapters. Chapter \ref{section PU(H)} is
devoted to understanding the equivariant projective unitary bundles
over a point, which are classified by the moduli space of homomorphisms $f: G \to
P\UU(\HH)$ from a compact Lie group  to the group of projective
unitary operators,  such that the induced action of the group
$\widetilde{G}:=f^*\UU(\HH)$ on $\HH$, makes $\HH$ into a
representation of $\widetilde{G}$ on which all its irreducible
representations where $S^1$ acts by multiplication, appear
infinitely number of times. It turns out that the equivariant homotopy groups
$\pi_i$ of this moduli space are isomorphic to the cohomology
groups $H^{3-i}(BG, \integer)$ for $i \geq 0$. We note that some
of the results of this chapter were already in \cite{AtiyahSegal},
but we believe that the main construction of this chapter, which
is the homotopy quotient
$$EP\UU(\HH) \times_{P\UU(\HH)} Hom_{st}(G, P\UU(\HH)),$$
where $Hom_{st}(G, P\UU(\HH))$ is the space of homomorphisms on
which all representations of $\widetilde{G}$ appear, has not been
studied in the literature before.

In Chapter \ref{section Equivariant stable bundles} we proceed to define the projective unitary stable
and equivariant bundles over $G$ spaces, and devote the rest of the chapter to the 
classification of the equivariant projective unitary stable bundles over the orbit types $G/K$ for 
$K$ compact subgroup of $G$. We show that the universal projective unitary bundle  over the orbit type $G/K$
can be constructed from the moduli space of stable homomorphisms from $K$ to $P\UU(\HH)$, and we furthermore
show that this model serves as the universal moduli space for $K$-equivariant projective unitary stable bundles over
trivial $K$-spaces. The gluing of these local universal bundles built out from stable homomorphisms
turns out very subtle and it is not at all clear how to do it. Therefore we propose larger models for the universal bundles of the
orbit types, models  that are built out of functors instead of homomorphisms; the advantage of these larger models is that they
are constructed in such a way that the gluing becomes very clear, but the disadvantage is that we have to restrict
our attention to discrete groups since these larger models may not have the desired properties in the general case.

In Chapter \ref{chapter universal bundle} we restrict our attention to the case on which $G$ is discrete and 
it acts properly. We devote the whole chapter to 
the construction of the universal projective unitary
stable and equivariant bundle; this is the main result of
this paper and can be found in Theorem \ref{theorem the universal bundle}. Using
stable functors from the groupoid $G \ltimes G/K$ to the group $P\UU(\HH)$, we construct a category that we denoted
$\widetilde{D}_{G/K}$  endowed with a free right $P\UU(\HH)$ action and a left $N(K)/K$ action, whose geometrical
realization is the universal equivariant projective unitary stable bundle over the orbit type $G/K$. Then we use the fact that
the categories $\widetilde{D}_{G/K}$ were defined from functors from $G \ltimes G/K$ in order to explicitely define the gluing of these
universal spaces, and therefore obtaining the universal projective unitary
stable and equivariant bundle. We finish this chapter by calculating the homotopy type of this universal space,
and in particular we show that 
the isomorphism classes of equivariant projective unitary
stable bundles are classified by the third
equivariant integral cohomology group..

Finally, in Chapter \ref{chapter Twisted equivariant K-theory} we define the twisted equivariant
K-theory groups for proper actions using the
projective unitary stable and equivariant bundles defined in
Chapter \ref{section Equivariant stable bundles}. We show that this twisted equivariant K-theory
satisfies the axioms of a generalized cohomology theory, that
satisfies Bott periodicity, that is endowed with an {\it{induction
structure}} as in \cite[Section 1]{Lueck1} and that restricted to
orbit types $G/K$ recovers the Grothendieck group
$R_{S^1}(\widetilde{G})$ of representations of $\widetilde{G}$ on
which $S^1$ acts by multiplication. We finish the paper by
relating our definition of twisted equivariant K-theory groups to
the definition of Dwyer in \cite{Dwyer} which is given through
projective representations characterized by discrete torsion.

\subsection{Notation:} \label{Notation} $G$ will be the global group and $K \subset
G$ will be a compact subgroup of $G$; $BG$ will be the classifying
space of $G$-principal bundles and it will be defined as $EG/G$
where $EG$ is the universal $G$-principal bundle. Letters in
calligraphic style will denote groupoids or categories; and for a
topological category $\CC$  we will denote by $| \CC |$ the
geometric realization of the category $\CC_{\bf N}$, that is, the
associated category unravelled over the ordered set ${\bf N}$ of
natural numbers in the following way: $\CC_{\bf N}$ is the
subcategory of ${\bf N} \times \CC$ obtained by deleting all
morphisms of the form $(n,c) \to (n,c')$ except for the identity
morphisms. Having defined the geometrical realization in this way
we have that $|G|$ is a the classifying space for principal $G$-bundles,
 and the map $|G\times G|\to |G|$ is a principal
$G$-bundle  where
$G\times G$ denotes the product category of the set $G$; for
details  see \cite[page 107]{Segal}.

\subsection{Acknowledgements}\label{ackref}
The first author was partially supported  by  the  SFB  878 `Groups, Geometry  and  Actions", and by the Leibniz prize of Prof. Dr. Wolfgang L\"uck.
The fourth author acknowledges and thanks the financial support of the 
Alexander Von Humboldt Foundation. Part of this work was carried out while the fourth author held a ``Humboldt
Research Fellowship for experienced researchers".

\section{Properties of the group $P\UU(\HH)$ of projective unitary
operators}\label{section PU(H)}
\subsection{Topology of $P\UU(\HH)$}

Let $\HH$ be a separable Hilbert space and 
$$\UU(\HH):= \{ U : \HH \to \HH | U\circ U^*= U^*\circ U = \mbox{Id} \}$$ the group
of unitary operators acting on $\HH$. Let ${\rm End}(\HH)$ denote the space of endomorphisms
of the Hilbert space and endow ${\rm End}(\HH)_{c.o.}$ with the compact open topology. Consider the inclusion
\begin{align*}
\UU(\HH) &\to {\rm End}(\HH)_{c.o.} \times {\rm End}(\HH)_{c.o.}\\
U &\mapsto (U,U^{-1})
\end{align*}
and induce on $\UU(\HH)$ the subspace topology. Denote the space of unitary operators with this induced topology
by $\UU(\HH)_{c.o.}$ and note that this is different from the usual compact open topology on $\UU(\HH)$.

It was pointed out in \cite[Appendix 1]{AtiyahSegal} that the group $\UU(\HH)_{c.o.}$ fails to be a topological group, since the composition
of operators is only continuous on compact subspaces. Knowing this, we can endow $\UU(\HH)$ with the compactly generated topology
induced by the topology of $\UU(\HH)_{c.o.}$ and in this way the composition of operators becomes continuous; denote by
 $\UU(\HH)_{c.g.}$ the group of unitary operators with the compactly generated topology induced by $\UU(\HH)_{c.o.}$ 
 
  Since in \cite[Prop. A2.1]{AtiyahSegal} there was constructed
 a homotopy $h: \UU(\HH)_{c.o.} \times[0,1] \to \UU(\HH)_{c.o.}$ such that $h(U,1)=U$ and $h(U,0)=constant$ with the property that $h$ is 
 continuous on compact sets, then the same map $h: \UU(\HH)_{c.g.} \times[0,1] \to \UU(\HH)_{c.g.}$ is continuous in the compactly generated topology and 
 therefore it proves the contractibility of the space $\UU(\HH)_{c.g.}$. Summarizing,
 \begin{proposition} \label{Proposition contractibility of U(H)}
 Let $\UU(\HH)_{c.g.}$ denote the group of unitary operators $\UU(\HH)$ endowed with the compactly generated topology of 
 $\UU(\HH)_{c.o.}$. Then $\UU(\HH)_{c.g.}$ is a contractible topological group.
 \end{proposition}
 \begin{definition}
 Let $P\UU(\HH)$ be the projectivization of $\UU(\HH)$ and endow it with the quotient topology of the quotient $\UU(\HH)_{c.g.}/S^1$.
  The topological group $P \UU(\HH)$ will be called the {\it group of
projective unitary operators} and it fits into the short exact
sequence of topological groups
$$1 \To S^1 \To \UU(\HH)_{c.g.} \To P\UU(\HH) \To 1.$$
\end{definition}

Since we know that $\UU(\HH)_{c.g.}$ is contractible,
then we have that the homotopy type of $P \UU(\HH)$ is the
one of an Eilenberg-Maclane space $K(\integer,2)$ and therefore
the underlying space of $P \UU(\HH)$ is a universal space for
complex line bundles.

\begin{rem}
The norm topology on $\UU(\HH)$ is an inadequate topology for the purpose of constructing equivariant projective unitary bundles.
For instance take $\HH=L^2(S^1)$ and define the translation action $(R_\lambda f)(\theta):= f( \theta \lambda^{-1})$ for $\lambda \in S^1$ and $f \in L^2(S^1)$. This action induces a homomorphism $S^1 \to \UU(\HH)$, $\lambda \mapsto R_\lambda$ which is not continuous whenever $\UU(\HH)$ is endowed with the norm topology; this follows from the fact that for $\lambda, \sigma \in S^1$ with $\lambda \neq \sigma$ we have $||R_\lambda - R_\sigma|| \geq {\rm{sup}}_{k\in \integer}\{|\lambda^k-\sigma^k|\} >1$.
\end{rem}

\subsection{$S^1$-central extensions}
The group structure on $P \UU(\HH)$ also permits to define $S^1$-central extensions of a group $G$ out of continuous homomorphisms
from $G$ to $P \UU(\HH)$ in the following way.

Let $G$ be a Lie group and $Ext(G, S^1)$ be the isomorphism
classes of $S^1$-central extensions of $G$ $$1 \to S^1 \to
\widetilde{G} \to G \to 1$$ where $\widetilde{G}$ also has the
structure of an $S^1$-principal bundle over $G$.

 Let us define the
map $\Psi$ from the space of continuous homomorphisms from $G$ to
$P\UU(\HH)$ endowed with the compact open topology, to the set of
isomorphism classes of $S^1$-central extensions of $G$
\begin{eqnarray*}
\Psi : Hom(G, P \UU(\HH))&  \to & Ext(G, S^1)\\
a : G \to P \UU(\HH) & \mapsto & \widetilde{G} : = a^* \UU(\HH),
\end{eqnarray*}
where $\widetilde{G}$ and $\widetilde{a}$ denote respectively the
Lie group and the continuous homomorphism defined by the pullback
square
$$\xymatrix{
\widetilde{G} \ar[r]^{\widetilde{a}} \ar[d] & \UU(\HH) \ar[d] \\
G \ar[r]^a & P\UU(\HH).}
$$

\begin{lemma} \label{lemma Psi surjective}
For $G$ a compact Lie group, the map $$\Psi : Hom(G, P \UU(\HH))
\to  Ext(G, S^1)$$ is surjective.
\end{lemma}

\begin{proof}
Let $\widetilde{G}$ be a $S^1$-central extension of $G$ and
consider the Hilbert space $L^2(\widetilde{G}) \otimes L^2([0,1])$
which is tensor product of square Hilbert space of integrable
functions on $\widetilde{G}$ and the Hilbert spaces of square
integrable functions on the unit interval. By Peter-Weyl's
theorem, $L^2(\widetilde{G}) \otimes L^2([0,1])$ contains all
irreducible representations of $\widetilde{G}$ infinitely number
of times. Now consider $V_c(\widetilde{G})$ to be the subspace of
$L^2(\widetilde{G}) \otimes L^2([0,1])$ of elements on which the
$S^1$-central of $\widetilde{G}$ acts by multiplication by
scalars. Take $\HH := V_c(\widetilde{G})$ and define the
homomorphism
$$\widetilde{a} : \widetilde{G} \to \UU(\HH), \ \ \ \ \ \ \widetilde{a}(g) v : = gv$$
induced by the left action of $\widetilde{G}$ on
$V_c(\widetilde{G})$; note that this homomorphism $\widetilde{a}$
is continuous because we have endowed $\UU(\HH)$ with the
compactly generated compact open topology. Taking the
projectivization
$$a : G \to P\UU(\HH)$$ of the map $\widetilde{a}$ we get the
desired homomorphism such that $$a^* \UU(\HH) \cong
\widetilde{G}.$$
\end{proof}

Clearly conjugate homomorphisms define isomorphic $S^1$-central
extensions, therefore we get an induced map
$$
 Hom(G, P \UU(\HH)) / P\UU(\HH)  \to  Ext(G, S^1).
$$

Now, take any two continuous homomorphisms $a,b : G \to P
\UU(\HH)$ such that their induced $S^1$-central extensions are
isomorphic, i.e. $\Psi(a) \cong \Psi(b)$. The maps $a$ and $b$ are
$P \UU(\HH)$-conjugate whenever the induced homomorphisms
$\widetilde{a} , \widetilde{b} : \widetilde{G} \to \UU(\HH)$ are
conjugate. The two maps $\widetilde{a}$ and $\widetilde{b}$ are
conjugate whenever the induced $\widetilde{G}$ representations on
$\HH$ can be written as the direct sum of the same number of
irreducible representations for each irreducible representation in
$V_c(\widetilde{G})$. Therefore we would like to focus our
attention to the $\widetilde{G}$ actions on $\HH$ such that all
irreducible representations in $V_c(\widetilde{G})$ appear
infinitely number of times.

\begin{definition} \label{definition stable}
A continuous homomorphism $a : G \to P\UU(\HH)$ is called { \bf
stable} if the unitary representation $\HH$ induced by the
homomorphism $\widetilde{a}: \widetilde{G}= a^*\UU(\HH) \to
\UU(\HH)$ contains each of the irreducible representations of
$\widetilde{G}$ that appear in $V_c(\widetilde{G})$ infinitely
number of times. We denote
$$Hom_{st}(G, P\UU(\HH)) \subset Hom(G,P\UU(\HH))$$
the subspace of continuous stable homomorphisms.
\end{definition}

\begin{proposition} \label{proposition pi zero of stable hom}
The map $\Psi$ induces a bijection of sets
$$Hom_{st}(G, P\UU(\HH))/P\UU(\HH) \stackrel{\cong}{\To} Ext(G,S^1).$$
between the set of isomorphism classes of continuous stable
homomorphisms and the set isomorphism classes of $S^1$-central extensions of $G$.
\end{proposition}

\begin{proof}

Let $\widetilde{G}$ be a $S^1$-central extension of $G$. In Lemma
\ref{lemma Psi surjective} we showed that if we take
$\HH:=V_c(\widetilde{G})$ the subspace of $L^2(\widetilde{G})
\otimes L^2([0,1])$ where $S^1$ acts by multiplication, then the
projectivization $a: G \to P\UU(\HH)$ of the induced homomorphism
$\widetilde{a}: \widetilde{G} \to \UU(\HH)$ is a continuous stable
homomorphism and it produces the desired $S^1$-central extension
$a^*\UU(\HH) \cong \widetilde{G}$; this shows surjectivity.

Let us suppose now that we have two continuous stable
homomorphisms $a,b$ such that $a^*\UU(\HH) \cong \widetilde{G}
\cong b^*\UU(\HH)$. The Hilbert space $\HH$ becomes a
$\widetilde{G}$ representation with respect to the map
$\widetilde{a}: \widetilde{G} \to \UU(\HH)$ and therefore there is
a (non-canonical) $\widetilde{G}$-equivariant isomorphism $f_a
:\HH \stackrel{\cong}{\to} V_c(\widetilde{G})$ that can be taken
to be unitary. We get the same result for the map $b$ and we get
another $\widetilde{G}$-equivariant isomorphism $f_b :\HH
\stackrel{\cong}{\to} V_c(\widetilde{G})$.

The $\widetilde{G}$-equivariant isomorphism $f_b^{-1} \circ f_a:
\HH \to \HH$ makes the following diagram commute
$$\xymatrix{
\HH \ar[rr]^{f_b^{-1} \circ f_a} \ar[d]_{\widetilde{a}(g)} & & \HH \ar[d]^{\widetilde{b}(g)}\\
\HH  \ar[rr]^{f_b^{-1} \circ f_a} & &\HH }$$
 for all $g \in \widetilde{G}$. Therefore the homomorphisms
$\widetilde{a}$ and $\widetilde{b}$ are conjugate; the injectivity
follows.

\end{proof}
\subsection{Groupoid of stable homomorphisms}
 \label{subsection homotopy quotient}

Let us now consider the action groupoid $$[Hom_{st}(G,
P\UU(\HH))/P\UU(\HH)]$$ associated to the action of the group
$P\UU(\HH)$ on the continuous stable homomorphisms by
conjugation; this groupoid can also by understood as the groupoid
whose objects are functors from the category defined by $G$ to the
category defined by $P\UU(\HH)$, and whose morphisms are natural
transformations.

We would like to understand the homotopy type of the classifying
space $$B[Hom_{st}(G, P\UU(\HH))/P\UU(\HH)],$$ which by our conventions of \S \ref{Notation}
and the arguments of \cite[Page 107]{Segal}, it can be
modelled by the homotopy quotient $$ EP\UU(\HH) \times_{P\UU(\HH)}
Hom_{st}(G, P\UU(\HH)).$$ 

By Proposition \ref{proposition pi zero of stable hom} we know
that its connected components are parametrized by the
$S^1$-central extensions of $G$,
$$\pi_0(EP\UU(\HH) \times_{P\UU(\HH)}
Hom_{st}(G, P\UU(\HH))) \cong Ext(G,S^1).$$ If we choose the
connected component of the classifying space determined by a
stable map $a: G \to P\UU(\HH)$, its higher homotopy groups are
determined by the homotopy groups of
$$ EP\UU(\HH) \times_{P\UU(\HH)_a} \{a\}  \simeq B
(P\UU(\HH)_a)$$ where
$$P\UU(\HH)_a =\{ b \in P\UU(\HH) | \forall g \in G, \ b^{-1}\, a(g)\, b=a(g) \}$$
is the subgroup of $P\UU(\HH)$ that stabilizes $a$.

Note that if we call $\widetilde{G}_a$ the $S^1$-central extension
$a^* \UU(\HH)$ defined by $a$, then the induced homomorphism
$\widetilde{a} : \widetilde{G}_a \to \UU(\HH)$ defines an action
on the Hilbert space $\HH$ making it isomorphic to the Hilbert
space $V_c(\widetilde{G}_a)$ defined in Lemma \ref{lemma Psi
surjective}. We can now take the action of $\widetilde{G}_a$ on
$\UU(\HH)$ by conjugation through the map $\widetilde{a}$; this
induces an action of the group $G$ on $P\UU(\HH)$ by conjugation
through the map $a$. Let us denote the stabilizer of the action $G$ on $P\UU(\HH)$
by conjugation through the map $a$ by $G_a$.

If we call the fixed points of the action of $G_a$ on $P\UU(\HH)$
by
$$P\UU(\HH)^{G_a}:= \{ b \in P\UU(\HH) |\forall g \in G, \ a(g)^{-1} \, b  \,
a(g) = b \}$$ then we have that
$$P\UU(\HH)^{G_a} =P\UU(\HH)_a.$$
This means that the fixed point set of the action of $G_a$ is the
same as the stabilizer of the homomorphism $a$.

Now take a projective operator $F \in P\UU(\HH)^{G_a}$ that
commutes with the $G_a$ action. Take $\widetilde{F} \in \UU(\HH)$
a lift of $F$ and note that $\widetilde{F}$ commutes with the
$\widetilde{G}_a$ action up to some phase, i.e. for all
$\widetilde{g} \in \widetilde{G}_a$ we have that
$$\widetilde{a}( \widetilde{g})^{-1} \, \widetilde{F} \, \widetilde{a}(
\widetilde{g}) =  \widetilde{F} \, \sigma_F( \widetilde{g})$$ for
some map $\sigma_F: \widetilde{G}_a \to S^1$. Because the action
of $\widetilde{G}_a$ is by conjugation, we have that the map
$\sigma_F$ is a homomorphism of groups, that it does not depend on
the choice of lift for $F$, and moreover that $\sigma_F$ is trivial
when restricted to $S^1= kernel(\pi: \widetilde{G}_a \to G_a).$ We
have then,
\begin{lemma}
For any $F \in P\UU(\HH)^{G_a}$ there exists a homomorphism
$\sigma_F \in Hom(G_a, S^1)$ such that for all lifts
$\widetilde{F} \in \UU(\HH)$ of $F$ and all $ \widetilde{g} \in
\widetilde{G}_a$, we have that
$$\widetilde{a}( \widetilde{g})^{-1} \, \widetilde{F} \, \widetilde{a}(
\widetilde{g}) = \widetilde{F}\,   \sigma_F(g)$$ where  $g$ is the
image of $\widetilde{g}$ on $G_a$ under the natural map
$\widetilde{G}_a \to G_a$.
\end{lemma}

If we take two operators $F_1, F_2 \in P\UU(\HH)^{G_a}$ with their
respective lifts $\widetilde{F}_1, \widetilde{F}_2$ and the
induced homomorphisms $\sigma_{F_1}, \sigma_{F_2} \in
Hom(G_a,S^1)$, we have that the composition $\widetilde{F}_1
\widetilde{F}_2$ is a lift for the composition ${F}_1 {F}_2$.
Therefore we have that the induced homomorphism $\sigma_{{F}_1
{F}_2}$ for the composition ${F}_1 {F}_2$ is equal to the product
of the homomorphisms $\sigma_{F_1}$ and $ \sigma_{F_2}$, i.e.
 $$\sigma_{{F}_1\, {F}_2} = \sigma_{F_1} \sigma_{F_2}.$$
Note that the product of homomorphisms endows the set $Hom(G_a,
S^1)$ with a natural group structure. We have then

\begin{lemma} \label{Lemma surjective sigma}
The map $$\sigma: P\UU(\HH)^{G_a} \to Hom(G_a,S^1), \ \ \ \ \ F
\mapsto \sigma_F$$ is a surjective homomorphism of groups.
\end{lemma}
\begin{proof}
We have already shown that $\sigma$ is a homomorphism of groups,
let us show that is surjective. Take $\rho \in Hom(G_a,S^1)$ and
consider $\complex(\rho)$ the linear $\widetilde{G}_a$
representation defined by multiplication of scalars, i.e. for
$\lambda \in \complex(\rho)$,  $\widetilde{g} \cdot \lambda =
\rho(g) \lambda$.

Now, if $V$ is an irreducible representation of $\widetilde{G}_a$
where the kernel of $\widetilde{G}_a \to G_a$ acts by
multiplication of scalars, the vector space $V \otimes
\complex(\rho)$ becomes a $\widetilde{G}_a$ representation by the
diagonal action, and moreover it is irreducible and the kernel
$\widetilde{G}_a \to G_a$ also acts by multiplication of scalars.
On the other hand, any irreducible representation of
$\widetilde{G}_a$, where the kernel of $\widetilde{G}_a \to G_a$
acts by multiplication of scalars, is isomorphic to a
representation of the form $W \otimes \complex(\rho)$ for a
suitable irreducible representation $W$ of $\widetilde{G}_a$ where
the kernel of $\widetilde{G}_a \to G_a$ acts by multiplication of
scalars. Therefore we have that by tensoring the space
$V_c(\widetilde{G}_a)$ with $\complex(\rho)$ we get a
$\widetilde{G}_a$-equivariant isomorphism $$ \eta:
V_c(\widetilde{G}_a) \otimes \complex(\rho) \to
V_c(\widetilde{G}_a)$$ where $V_c(\widetilde{G}_a)$ is the Hilbert
space defined in Lemma \ref{lemma Psi surjective}.

Let us consider the map
$$\gamma : V_c(\widetilde{G}_a) \to V_c(\widetilde{G}_a) \otimes
\complex(\rho), \ \ \ \ \ \gamma(v)= v \otimes 1$$ and note that
$$\widetilde{g} \gamma ( \widetilde{g}^{-1} v) = \widetilde{g}
\left( \gamma (\widetilde{g}^{-1} v) \otimes 1 \right) = v \otimes
\rho(g) = \rho(g) \gamma(v).$$

Choosing an explicit $\widetilde{G}_a$-equivariant isomorphism
$f_a : \HH \to V_c(\widetilde{G}_a)$ such as the one defined in
Proposition \ref{proposition pi zero of stable hom}, we can define
the operator in $\UU(\HH)$
$$\widetilde{F} : = f_a^{-1} \circ \eta \circ \gamma \circ f_a$$
which by definition has the property that for all
$\widetilde{g} \in \widetilde{G}_a$ \begin{equation}
\label{equation equivariant lift up to a phase} \widetilde{a}(
\widetilde{g})^{-1} \, \widetilde{F} \, \widetilde{a}(
\widetilde{g}) = \widetilde{F} \,  \sigma_F(g).\end{equation}

If $F$ is the projectivization of $\widetilde{F}$, then we have
that equation \eqref{equation equivariant lift up to a phase}
implies that $F \in P \UU(\HH)^{G_a}$, and moreover that $\sigma_F
= \rho$. This shows that the homomorphism $\sigma$ is surjective.
\end{proof}

With the help of the surjective map $\sigma$ let us try to
understand in a more explicit way the group $P\UU(\HH)^{G_a}$.
Take $\rho \in Hom(G_a,S^1)$ and define the space
$$\UU(\HH)^{(\widetilde{G}_a,\rho)} := \{ \widetilde{F} \in
\UU(\HH) | \forall \widetilde{g} \in \widetilde{G}_a, \
\widetilde{a}(\widetilde{g})^{-1} \, \widetilde{F} \,
\widetilde{a}( \widetilde{g}) = \widetilde{F} \,  \sigma_F(g)\}.$$
From the proof of Lemma \ref{Lemma surjective sigma} we get that
the inverse image of $\rho$ under $\sigma$ is isomorphic to the
projectivization of the space $\UU(\HH)^{(\widetilde{G}_a,\rho)}$,
i.e.
$$ \sigma^{-1}(\rho) = P \left( \UU(\HH)^{(\widetilde{G}_a,\rho)}\right),$$
and moreover we have that any  operator $\widetilde{F} \in
\UU(\HH)^{(\widetilde{G}_a,\rho)}$ induces a homeomorphism of
spaces
$$\beta_{\widetilde{F}} : \UU(\HH)^{\widetilde{G}_a}
\stackrel{\cong}{\To} \UU(\HH)^{(\widetilde{G}_a,\rho)}, \ \ \ W
\mapsto F \circ W $$ by composing the operators.

\begin{theorem} \label{theorem homotopy groups of PU(H)G}
Let $a \in Hom_{st}(G, P\UU(\HH))$ be a continuous stable
homomorphism, and denote with $G_a$ the action on $P\UU(\HH)$ by
conjugation defined by the map $a$. Then there is a natural
isomorphism of groups
$$\pi_0\left(P\UU(\HH)^{G_a} \right) \cong Hom(G;S^1),$$
and each connected component of $P\UU(\HH)^{G_a}$ has the homotopy
type of a $K(\integer,2)$.
\end{theorem}

\begin{proof}
We have seen from the arguments above that
$$P\UU(\HH)^{G_a} \cong \bigsqcup_{\rho \in Hom(G,S^1)} P \left(
\UU(\HH)^{(\widetilde{G}_a,\rho)}\right)$$ and that as spaces we
have homeomorphisms
$$P \left( \UU(\HH)^{\widetilde{G}_a}\right) \cong P \left(
\UU(\HH)^{(\widetilde{G}_a,\rho)}\right).$$

The  ${\widetilde{G}_a}$ action on $\HH$ splits the Hilbert space $\HH$ into  isotypical $\widetilde{G}_a$-representations
 $$\HH = \oplus_\alpha \HH_\alpha$$
 with $\HH_\alpha \cong M_\alpha \otimes \HH_{\alpha,0}$ where $M_\alpha$ is a simple $\widetilde{G}_a$-representation
and $\HH_{\alpha,0}$ is a separable Hilbert space. Therefore we have 
$$\UU(\HH)^{\widetilde{G}_a} \cong  \prod_\alpha \UU(\HH_\alpha)^{\widetilde{G}_a}  \cong \prod_\alpha \UU(\HH_{\alpha,0})$$
where the stability condition of the homomorphism $a$ implies that each of the Hilbert space spaces $\HH_{\alpha,0}$ is infinite dimensional and hence the spaces $ \UU(\HH_{\alpha,0})$ are all contractible by Proposition \ref{Proposition contractibility of U(H)}. 
Then we can conclude that $ \UU(\HH)^{\widetilde{G}_a}$ is contractible and therefore we have that
$$\pi_i\left(P\left (\UU(\HH)^{\widetilde{G}_a}\right) \right) = \left\{
\begin{array}{cl}
\integer & i =0,2\\
0 & \mbox{otherwise.}
\end{array} \right.$$
Therefore the spaces $P \left
(\UU(\HH)^{(\widetilde{G}_a,\rho)}\right)$
 have the homotopy type of a $K(\integer,2)$,
and moreover, the connected components of $P\UU(\HH)^{G_a}$ are
parameterized by the elements in $Hom(G_a,S^1)$; hence we have an
isomorphism of groups
$$\pi_0\left(P\UU(\HH)^{G_a} \right) \cong Hom(G;S^1).$$

\end{proof}

We can now bundle up all previous results in the following theorem
\begin{theorem} \label{theorem homotopy groups BPU(H)}
Let $G$ be a compact Lie group and
$Hom_{st}(G,P\UU(\HH))/P\UU(\HH)$ the category whose space of
objects consist  of continuous stable homomorphisms from $G$ to
the projective unitary group endowed with the compact open
topology, and whose morphisms are natural transformations. Then
the connected components of the homotopy quotient are
parameterized by the $S^1$-central extensions of $G$,
$$\pi_0\left(EP\UU(\HH) \times_{P\UU(\HH)} Hom_{st}(G,P\UU(\HH)) \right) = Ext_c(G,S^1),$$
and the higher homotopy groups of any connected component are
$$\pi_i\left(EP\UU(\HH) \times_{P\UU(\HH)} Hom_{st}(G,P\UU(\HH))\right)=\left\{
\begin{array}{cl}
Hom(G,S^1) & i=1\\
\integer & i=3\\
0 & \mbox{otherwise}.
\end{array}
\right.$$
\end{theorem}

\begin{proof}
From the analysis done at the beginning of Section \ref{subsection
homotopy quotient} we know that the connected components of the
homotopy quotient $$EP\UU(\HH) \times_{P\UU(\HH)}
Hom_{st}(G,P\UU(\HH))$$ are parameterized by $Ext_c(G,S^1)$. Also
from Section \ref{subsection homotopy quotient}, we know that if
we take any stable homomorphism $a:G \to P\UU(\HH)$, the connected
component defined by $a$ is homotopy equivalent to $$EP\UU(\HH)
\times_{P\UU(\HH)_a} \{a\} \cong B (P\UU(\HH)_a ) = B(
P\UU(\HH)^{G_a} ).$$

By Theorem \ref{theorem homotopy groups of PU(H)G} we know that $\pi_0(P\UU(\HH)^{G_a})=Hom(G,S^1)$
is an isomorphism of groups, and therefore we have that
$$\pi_1\left(B( P\UU(\HH)^{G_a} )\right) = Hom(G,S^1).$$
Also from Theorem \ref{theorem homotopy groups of PU(H)G} we know
that $\pi_2(P\UU(\HH)^{G_a})= \integer$ and that the other
homotopy groups are trivial. Therefore
$$\pi_3\left(B( P\UU(\HH)^{G_a} )\right)= \integer$$ and
$$\pi_i\left(B( P\UU(\HH)^{G_a} )\right)= 0$$
for $i=2$ and $i>3$.
\end{proof}

The homotopy groups described previously are precisely the
cohomology groups for the classifying space $BG$ in certain
degree, let us see:

\begin{cor} \label{corollary equivalence of homotopy type for G}
For $G$ a compact Lie group there are isomorphisms
$$\pi_i\left(EP\UU(\HH) \times_{P\UU(\HH)} Hom_{st}(G,P\UU(\HH)) \right) =
H^{3-i}(BG, \integer).$$
\end{cor}

\begin{proof}

The action of $G$ on $EP\UU(\HH) \times Hom_{st}(G,P\UU(\HH))$
given by equation (\ref{left K action}) defines a projective
unitary bundle on the homotopy quotients
$$\xymatrix{
P\UU(\HH) \ar[r] & (EP\UU(\HH) \times Hom_{st}(G,P\UU(\HH)))
\times_G EG \ar[d]& \\
& EP\UU(\HH) \times_{P\UU(\HH)} Hom_{st}(G,P\UU(\HH)) \times BG &
}$$ classified by a map from the base to $BP\UU(\HH)$, which
induces an adjoint map
\begin{eqnarray} \label{induced map to BG}
EP\UU(\HH) \times_{P\UU(\HH)} Hom_{st}(G,P\UU(\HH)) \to Maps(BG,
BP\UU(\HH)). \end{eqnarray}

At the level of homotopy groups we know that
$$\pi_i(Maps(BG, BP\UU(\HH)) \cong H^{3-i}(BG, \integer)$$ and
therefore we get the desired map
$$\pi_i\left(EP\UU(\HH) \times_{P\UU(\HH)} Hom_{st}(G,P\UU(\HH)) \right)
\to H^{3-i}(BG, \integer).$$

The space $BG$ is connected and therefore $H^0(BG, \integer)=
\integer$, agreeing with the third homotopy group of the homotopy
quotient.

The group $G$ being compact implies that $\pi_0(G)$ is finite.
Hence $\pi_1(BG)$ is also finite and we have that
$H^1(BG,\integer)=0$; this shows the isomorphism for $i=2$.

In \cite[Prop. 4]{LashofMaySegal} it is proven that the natural
map
$$Hom(G,A) \to [BG, BA]$$
is an isomorphism whenever $G$ and $A$ are compact Lie groups and
$A$ is abelian. Therefore
$$Hom(G, S^1) \stackrel{\cong}{\to} [BG,BS^1] = H^2(BG,
\integer);$$ this shows the isomorphism for $i=1$.

Finally, in \cite[Prop. 6.3]{AtiyahSegal} Atiyah and Segal prove
that $H^3(BG, \integer) \cong Ext_c(G,S^1)$. This, together with
the fact that the connected component of $[(a, f)]$ in the left
hand side of (\ref{induced map to BG}), maps to the connected
component of $Bf: BG \to BP\UU(\HH)$ imply that the isomorphism
holds for $i=0$.
\end{proof}

\section{Equivariant stable projective unitary bundles and their
classification} \label{section Equivariant stable bundles}

This chapter is devoted to set-up the framework for generalizing
the twistings of K-theory obtained by principal
$P\UU(\HH)$-bundles to the equivariant case. We start the
chapter by giving the definition of an {\it equivariant stable
projective unitary bundle}, we study the local objects, i.e. the 
equivariant stable
projective unitary bundles over the spaces $G/K$,  and then we proceed to show 
how they can be classified.

Let us begin by clarifying which type of spaces we will be working
with.

\subsection{Preliminaries}

\subsubsection{} Throughout this chapter $G$ will be a Lie group  and $X$ will be a proper $G-ANR$; let us
recall what all this means:

A  $G$-space $X$ is {\it proper} if the map
$$G \times X \to X \times X \ \ \ \ \ (g,x) \mapsto (gx,x)$$
is a proper map of topological spaces (preimages of compact sets are
compact) and the action map $G \times X \to X, \ (g,x) \mapsto gx$
is a closed map.

A $G$ space $X$ is a $G-ANR$ ($G$-equivariant absolute
neighborhood retract \cite{JamesSegal}) if $X$ is a separable and
metrizable $G$-space such that whenever $B$ is a normal $G$-space
and $A$ is an invariant closed subspace of $B$, any $G$ map $A \to
X$ can be extended over an invariant neighborhood of $A$ in $B$.
In the case that $G$ is compact Lie group acting smoothly on a compact
manifold $M$, then it can be shown that $M$ is always a $G-ANR$.

Note that any proper $G-ANR$ space can be covered with a {\it good
$G$-cover}, namely a countable $G$-cover such that any non trivial
intersection of finite $G$-open sets of the cover is equivariantly
homotopicaly equivalent to a $G$-set of the form $G/H$ with $H$ a
compact subgroup of $G$.

When studying the homotopy theory of $G$-spaces is important to
understand the homotopy theory of the system of fixed points of
the action. This is done with the category of canonical orbits,
let us recall its definition and some of its properties
 \cite[Section 7]{DavisLueck}.

\subsubsection{System of fixed points} \label{subsection system of
fix points}

The category of canonical orbits for proper $G$-actions denoted by
$\OO_G^P$ is a topological category with discrete object space
$${\rm{Obj}}(\OO_G^P) = \{G/H \colon H \ \mbox{is a compact
subgroup of} \ G \}$$ and whose morphisms consist of  $G$-maps
$${\rm Mor}_{\OO^P_G} (G/H,G/K) = {\rm Maps}(G/H,G/K)^G$$
with a topology such that the natural bijection
$${\rm Mor}_{\OO^P_G} (G/H,G/K) \cong (G/K)^H$$
is a homeomorphism.

An {\it $\OO_G^P$-space} is a contravariant functor from
$\OO_G^P$ to the category of topological spaces, and these
functors will form the objects of a topological category whose
morphisms will consist of natural transformations.

The {\it fixed point set system} of $X$, denoted by $\Phi X$, is
the $\OO_G^P$-space defined by:
$$\Phi X (G/H) := {\rm{Maps}}(G/H, X)^G= X^H$$
and if $\theta: G/H \to G/K$ corresponds to $gK \in (G/K)^H$ then
$$\Phi X (\theta)(x) := gx \in X^H$$
whenever $x \in X^K$. The functor $\Phi$ becomes a functor from
the category of proper $G$-spaces to the category of
$\OO_G^P$-spaces.

If $\XX$ is a contravariant functor from $\OO_G^P$ to spaces and $\YY$ is a covariant functor from $\OO_G^P$
to spaces one can define the space
$$\XX \times_{\OO_G^P}\YY:= \bigsqcup_{c \in {\rm{Obj}}(\OO_G^P)} \XX(c) \times \YY(c) / \sim$$
where $\sim$ is the equivalence relation generated by $(\XX(\phi)(x),y) \sim (x, \YY(\phi)(y))$
 for all morphisms $\phi: c \to d$ in $\OO_G^P$ and points $x \in \XX(d)$ and $y \in \YY(c)$.

A  model  for  the  homotopical  version  of the previous construction is defined as follows: consider the category $\FF(\XX, \YY)$
whose objects are
$${\rm{Obj}}(\FF(\XX, \YY))=\bigsqcup_{c \in {\rm{Obj}}(\OO_G^P)} \XX(c) \times \YY(c)$$
and whose morphisms consist of all triples $(x, \phi,y)$ where  $\phi: c \to d$ is a morphism in $\OO_G^P$ and
$x \in \XX(d)$ and $y \in \YY(c)$, with $\source(x, \phi,y)= (\XX(\phi)(x),y)$ and $\target(x, \phi,y)= (x, \YY(\phi)(y))$. 
Define the space
$$\XX \times_{h\OO_G^P}\YY:= |\FF(\XX, \YY)| $$
as the geometric realization of the category $\FF(\XX, \YY)$.

Now, if we consider  the covariant functor $\nabla$ from $\OO_G^P$ to spaces that to each orbit type $G/H$ it assigns
the set $G/H$, then we can consider the functors
\begin{eqnarray*}
  \_ \times_{\OO_G^P}\nabla :  \OO_G^P-{\rm spaces} & \to & {\rm proper} \ G-{\rm spaces} \\
  \XX & \mapsto & \XX \times_{\OO_G^P}\nabla\\
   \_ \times_{h\OO_G^P}\nabla :  \OO_G^P-{\rm spaces} & \to & {\rm proper} \ G-{\rm spaces} \\
  \XX & \mapsto & \XX \times_{h\OO_G^P}\nabla.
\end{eqnarray*}

Since the orbit types $G/H$ are endowed with a natural left $G$ action, then the spaces
$\XX \times_{\OO_G^P}\nabla$ and $\XX \times_{h\OO_G^P}\nabla$ become $G$-spaces by endowing them 
with the left $G$ action included by the $G$ action on the objects and morphisms of  $\nabla$.

We shall quote the following result:

\begin{theorem}  \label{theorem Davis-Lueck}
For $G$ a discrete group \cite[lemma  7.2]{DavisLueck} the functors $\Phi$ and $ \_ \times_{\OO_G^P}\nabla$
are adjoint, i.e. for a $\OO_G^P$ space $\XX$ and a proper $G$-space $Y$ there is a natural homeomorphism
$${\rm{Maps}}(\XX \times_{\OO_G^P}\nabla,Y)^G \stackrel{\cong}{\To} {\rm{Hom}}_{\OO_G^P} (\XX, \Phi Y),$$
and moreover, the adjoint of the identity map on $\Phi Y$ under the above adjunction, is a  natural $G$-homeomorphism
$$(\Phi Y) \times_{\OO_G^P}\nabla \stackrel{\cong}{\To} Y.$$
 \end{theorem}
 
 Note furthermore that there is a natural  $\OO_G^P$-homotopy equivalence
$$  \XX \to \Phi(\XX \times_{h\OO_G^P}\nabla).$$ Similar statements in the case that $G$ is a compact Lie group can be found in
\cite{Elmendorf}.

\subsection{Equivariant stable projective unitary bundles}

A first guess may suggest that the appropriate equivariant
twistings for K-theory from a point of view of index theory are
$G$-equivariant projective unitary bundles. These bundles may be
used to define associated bundles whose fibers are $Fred(\HH)$ the
space of Fredholm operators on $\HH$, but a closer look at them
leads us to think that we need to put further conditions on how
the group $G$ acts on the fibers; this local condition is that the
local isotropy group should act on the fibers by stable
homomorphisms. The precise definition is as follows:

\begin{definition}\label{def projective unitary G-equivariant stable bundle}
A {\it projective unitary $G$-equivariant stable bundle} over $X$
is a principal $P\UU(\HH)$-bundle
$$P\UU(\HH) \To P \To X$$
where $P\UU(\HH)$ acts on the right, endowed with a left $G$
action lifting the action on $X$ such that:
\begin{itemize}
\item the left $G$-action commutes with the right $P\UU(\HH)$
action, and \item for all $x \in X$ there exists a
$G$-neighborhood $V$ of $x$ and a $G_x$-contractible slice $U$ of
$x$ with $V$ equivariantly homeomorphic to $ U \times_{G_x} G$
with the action $$G_x \times (U \times G) \to U \times G, \ \ \ \
k \cdot(u,g)= (ku, g k^{-1}),$$ together with a local
trivialization
$$P|_V \cong  (P\UU(\HH) \times U) \times_{G_x} G$$ where the action of the isotropy group
is:
 \begin{eqnarray*}
 G_x \times \left( (P\UU(\HH) \times U) \times G \right)& \to & (P\UU(\HH) \times U) \times G
 \\
\ \left(k , ((F,y),g)\right)& \mapsto & ((f_x(k)F, ky), g k^{-1})
\end{eqnarray*} with $f_x : G_x \to P\UU(\HH)$ a fixed stable
homomorphism (see Definition \ref{definition stable}).
\end{itemize}

\end{definition}

Two projective unitary $G$-equivariant stable bundles $P', P$ over
$X$ will be isomorphic if there is a $G$-equivariant homeomorphism
$P' \to P$ of principal $P\UU(\HH)$ bundles. The isomorphism
classes of projective unitary $G$-equivariant stable bundles over
$X$ will be denoted by
$$\Bun^G_{st}(X, P\UU(\HH)).$$

\subsubsection{} The results of Chapter \ref{section PU(H)} will
provide us with the first examples of projective unitary
equivariant stable bundles. These examples will be the building
blocks of the construction of the universal projective unitary
$G$-equivariant stable bundle as well as their classification;
this will be done in Chapter \ref{chapter universal bundle} for the case of discrete and proper actions.

Let $K$ be a compact Lie group and consider the projective unitary
bundle \begin{equation} \label{projective unitary bundle K
compact}
\xymatrix{
P\UU(\HH) \ar[r] & EP\UU(\HH) \times Hom_{st}(K, P\UU(\HH)) \ar[d] \\
& EP\UU(\HH) \times_{P\UU(\HH)} Hom_{st}(K, P\UU(\HH)) }
\end{equation} where the base is the space whose homotopy groups
are calculated in Theorem \ref{theorem homotopy groups BPU(H)}.

The right-$P\UU(\HH)$ action on the total space of the bundle in
(\ref{projective unitary bundle K compact}) is defined as:
\begin{eqnarray}
 \nonumber EP\UU(\HH) \times Hom_{st}(K, P\UU(\HH)) \times P\UU(\HH) & \to &
EP\UU(\HH) \times Hom_{st}(K, P\UU(\HH))\\
((a,f),F) & \mapsto & (aF, F^{-1} f F) \label{right PU(H) action}
\end{eqnarray}
where $F^{-1} f F$ denotes the homomorphism  conjugate of $f$ by
$F$.

The left $K$-action on the total space of (\ref{projective unitary
bundle K compact}) is defined as:
\begin{eqnarray}
\nonumber K \times EP\UU(\HH) \times Hom_{st}(K, P\UU(\HH)) & \to
&
EP\UU(\HH) \times Hom_{st}(K, P\UU(\HH))\\
(g,(a,f)) & \mapsto & (a f(g), f(g)^{-1} f f(g)), \label{left K
action}
\end{eqnarray}
where a simple calculation shows that the $K$-action is indeed a
left action.

It follows that the total space  of the bundle in (\ref{projective
unitary bundle K compact}) has a left $K$-action inducing a
trivial $K$ action on the base. We claim that

\begin{proposition}
The projective unitary bundle
$$\xymatrix{
P\UU(\HH) \ar[r] & EP\UU(\HH) \times Hom_{st}(K, P\UU(\HH)) \ar[d] \\
& EP\UU(\HH) \times_{P\UU(\HH)} Hom_{st}(K, P\UU(\HH)) }$$ is a
projective unitary $K$-equivariant stable bundle.
\end{proposition}

\begin{proof}
Let us first prove that the left $K$-action commutes with the
right $P\UU(\HH)$ action; on the one side we have
\begin{eqnarray*}
\left( g(a,f) \right) F& =& \left( a f(g), f(g)^{-1}f f(g))
\right) F \\
& = & \left(a f(g) F, F^{-1} f(g)^{-1} f f(g) F \right)
\end{eqnarray*}
and on the other
\begin{eqnarray*}
g\left( (a,f)F \right) & =& \left( a F, F^{-1}f F)
\right) F \\
& = & \left(a F F^{-1} f(g) F, F^{-1} f(g)^{-1}F F^{-1} f F F^{-1}
f(g) F\right) \\
& = & \left(a f(g) F, F^{-1} f(g)^{-1} f f(g) F \right);
\end{eqnarray*}
which shows that the actions commute.

Now, for the second condition let us take any point $x \in
EP\UU(\HH) \times_{P\UU(\HH)} Hom_{st}(K, P\UU(\HH))$ and let us
choose a contractible neighborhood $V$ of $x$. Because the
restricted bundle $$ \left( EP\UU(\HH) \times Hom_{st}(K,
P\UU(\HH))\right)|_V $$ is trivializable we may find a section
\begin{eqnarray} \label{diagram section to EPxHom}\xymatrix{
& EP\UU(\HH) \times Hom_{st}(K, P\UU(\HH)) \ar[d] \\
V \ar[ur]^\alpha \ar@{->}[r] & EP\UU(\HH) \times_{P\UU(\HH)}
Hom_{st}(K, P\UU(\HH))  }\end{eqnarray} which can be decomposed as
$$\alpha(v) = \left(\lambda(v), \eta(v) \right)$$ with $\lambda: V
\to EP\UU(\HH)$ and $\eta: V \to Hom_{st}(K, P\UU(\HH))$.

Let us denote $f := \eta(x)$ and let us consider the connected
component $W_f \subset Hom_{st}(K, P\UU(\HH))$ containing $f$. In
view of Proposition \ref{proposition pi zero of stable hom} we
know that the group $P\UU(\HH)$ acts transitively on $W_f$ and
therefore we have a non canonical homeomorphism
$$ P\UU(\HH)/P\UU(\HH)_f \stackrel{\cong}{\To} W_f, \ \ \ \ \ [F] \mapsto F^{-1}f F $$
where $$P\UU(\HH)_f= \{F \in P\UU(\HH) | F^{-1}fF = f \}$$ is the
isotropy group of $f$ and acts on $P\UU(\HH)$ by conjugation on
the right.

Because the bundle $$P\UU(\HH)_f \To P\UU(\HH) \To
P\UU(\HH)/P\UU(\HH)_f$$ is a principal bundle and $V$ is a
contractible set, we may find a lift $\sigma: V \to P\UU(\HH)$ of
the map $\eta: V \to W_f$ making the diagram commutative
$$\xymatrix{
&  & P\UU(\HH) \ar[d] \\
V \ar[r]_\eta  \ar[urr]^\sigma & W_f  \ar[r]^(.3)\cong &
P\UU(\HH)/P\UU(\HH)_f }$$ Hence the map $\sigma$ satisfies the
equation
$$\eta(v) = \sigma(v)^{-1} f \sigma(v)$$
for $v \in V$.

Let us now consider a different section $$\alpha' : V \to
EP\UU(\HH) \times Hom_{st}(K, P\UU(\HH))$$ of the diagram
(\ref{diagram section to EPxHom}) defined by the action of
$\sigma$ on $\alpha$, namely:
$$\alpha'(v) := \alpha(v) \cdot \sigma(v)^{-1} =
\left(\lambda(v)\sigma(v)^{-1}, \sigma(v) \eta(v) \sigma(v)^{-1}
\right) = (\lambda'(v), f)$$ where we denoted
$\lambda'(v):=\lambda(v) \sigma(v)^{-1}$.

With the section $\alpha'$ at hand we can define a local
trivialization as follows
\begin{eqnarray}
\nonumber V \times P\UU(\HH) & \stackrel{\phi}{\To} & \left(
EP\UU(\HH)
\times_{P\UU(\HH)} Hom_{st}(K,P\UU(\HH)) \right)|_V \\
(v,F) & \mapsto & \alpha'(v) \cdot F = \left( \lambda'(v) F,
F^{-1} f F \right). \label{local trivialization}
\end{eqnarray}
Transporting the $K$-action to the left-hand side of (\ref{local
trivialization}) we have that for $g \in K$ and $(v,F) \in V
\times P\UU(\HH)$,
\begin{eqnarray*}
g\cdot (v, F) & := & \phi^{-1} \left( g ( \phi(v, F))\right)\\
&=& \phi^{-1} \left( g( \lambda'(v)F, F^{-1}fF) \right)\\
&=& \phi^{-1}\left( \lambda'(v)f(g)F, F^{-1}f(g)^{-1} f f(g) F
\right) \\
&=& \left( v,f(g)F\right),
\end{eqnarray*}
which implies that the $K$ action on $P\UU(\HH)$ is by
multiplication on the left by the fixed stable homomorphism $f$.

This finishes the proof.

\end{proof}

\subsection{Local objects}

Following the notation of \cite{tomDieck, Murayama} we say that:
\begin{definition}
For $K$ a compact subgroup of $G$, a projective unitary stable
$G$-equivariant bundle over $G/K$ is a {\it{local object}}.
\end{definition}

We would like to classify the local objects. Let $P \to G/K$ be a
local object and consider the restriction $P|_{[K]}$ of $P$ to the
point $[K] \in G/K$. The canonical map $$P|_{[K]} \times_K G \to
P, \ \ \ \ \ [(x,g)] \mapsto gx$$ produces  a $G$-equivariant
isomorphism of projective unitary bundles, and by Definition
\ref{def projective unitary G-equivariant stable bundle} we know
that $P|_{[K]}$ trivializes via a map $\phi: P\UU(\HH) \cong
P|_{[K]}$ where the left $K$ action on $P|_{[K]}$ induces a $K$
action on $P\UU(\HH)$ via a stable homomorphism $f \in Hom_{st}(K,
P\UU(\HH))$.

We can take the $K$ action on $P\UU(\HH) \times G$ given by $k
\cdot (F,g) := (f(k)F, gk^{-1})$ and we obtain an isomorphism
$$P\UU(\HH) \times_K G \stackrel{\cong}{\to} P, \ \ \ \ \ [(F,g)] \mapsto g\phi(F)$$
of projective unitary stable $G$-equivariant bundles.

We can conclude that \begin{lemma} The isomorphism classes of
local objects over $G/K$ are in one to one correspondence with
conjugacy classes of stable homomorphisms from $K$ to $P\UU(\HH)$,
which in view of Proposition \ref{proposition pi zero of stable
hom} is isomorphic to the set of isomorphism classes of  $S^1$-central extensions of $K$.
\end{lemma}

We have classified the projective unitary stable $G$-equivariant
bundles over the $G$-space $Y \times_K G$ when $Y$ is a point. In
the next section we will generalize the classification whenever
$Y$ is a trivial $K$-space.

\subsection{Universal projective unitary stable $K$-equivariant bundle over trivial
$K$-spaces}\label{section universal for trivial K spaces}  We will
show that the universal projective unitary stable $K$-equivariant
is bundle precisely the bundle
\begin{eqnarray} \label{universal K bundle}
\xymatrix{
P\UU(\HH) \ar[r] & EP\UU(\HH) \times Hom_{st}(K,P\UU(\HH)) \ar[d] \\
&EP\UU(\HH) \times_{P\UU(\HH)} Hom_{st}(K,P\UU(\HH)).}
\end{eqnarray}
To carry out the proof we will use two groupoids $\DD_K$ and
$\CC_K$ and a functor between the two $\DD_K \to \CC_K$ such that
the geometric realization of the functor $| \DD_K| \to |\CC_K|$
is homotopy equivalent to the bundle in (\ref{universal K
bundle}). The groupoid $\CC_K$ will be the action groupoid
$$\CC_K:=Hom_{st}(K, P\UU(\HH)) \rtimes P\UU(\HH)$$ whose objects
are the stable homomorphisms and whose space of morphisms is
$$Hom_{st}(K, P\UU(\HH)) \times P\UU(\HH),$$
together with the structural maps
$$\source(f,F)=f, \ \ \ \target(f,F)=F^{-1}fF, \ \ \ \inverse(f,F)=(F^{-1}fF, F^{-1}),$$
$$\composition((f,F),(F^{-1}fF,G))=(f,FG) \ \ \ {\rm{and}} \ \ \identity(f) = (f,1).$$

\vspace{0.3cm}

The groupoid $\DD_K$ will consist of
$${\rm{Mor}}(\DD_K) := Hom_{st}(K, P\UU(\HH)) \times P\UU(\HH)  \times P\UU(\HH)$$
$${\rm{Obj}}(\DD_K):=Hom_{st}(K, P\UU(\HH)) \times P\UU(\HH)$$
together with structural maps
$$\source(f,F,G)=(f,F), \ \ \ \target(f,F,G)=(GF^{-1}fFG^{-1},G),$$
$$\inverse(f,F,G)=(GF^{-1}fFG^{-1}, G,F),$$
$$\composition((f,F,G),(GF^{-1}fFG^{-1},G,L))=(f,F,L) $$
$${\rm{and}} \ \ \identity(f,F) = (f,F,1).$$

The map $\gamma:{\rm{Mor}}(\DD_K) \to {\rm{Mor}}(\CC_K)$, $(f,F,G)
\mapsto (f,FG^{-1})$ induces a functor $\gamma:\DD_K \to \CC_K$,
and therefore we have a map $|\gamma|:|\DD_K| \to |\CC_K|$ at the
level of their geometric realizations. The results of
\cite[Section 3]{Segal} show that $|\DD_K|
\stackrel{|\gamma|}{\to} |\CC_K|$ is indeed projective unitary
bundle
$$P\UU(\HH) \to |\DD_K| \to |\CC_K|$$
and also that it is homotopy equivalent to the bundle in
(\ref{universal K bundle}); recall from the notation, section
\ref{Notation}, that in this paper the geometrical realization of
a category $\WW$ is defined as the geometrical realization of the
category $\WW_{\bf N}$.

Endowing the groupoid $\DD_K$ with the following left action of
the group $K$
$$K \times \DD_K \to \DD_K, \ \ \ \ (g,(f,F,G)) \mapsto (f,f(g)F, GF^{-1}f(g)F),$$
a simple calculation shows that the induced action on $\CC_K$ is
trivial. Hence we have
\begin{lemma}
The bundle $P\UU(\HH) \to |\DD_K| \stackrel{|\gamma|}{\to}
|\CC_K|$ is a projective unitary stable $K$-equivariant bundle
over the trivial $K$ space $|\CC_K|$.
\end{lemma}

Now let us do the main construction of this section
\begin{theorem} \label{theorem classifying map for trivial K spaces}
Let $P\UU(\HH) \to Q \to Y$ be a projective unitary stable
$K$-equivariant bundle over the trivial $K$-space $Y$, then there
is a map $\alpha:Y \to |\CC_K|$ such that $\alpha^*|\DD_K| \cong
Q$ as projective unitary stable $K$-equivariant bundles.
\end{theorem}

\begin{proof}

Choose a cover $\{U_i\}_{i \in I}$ of $Y$ where $Q$ is trivialized, i.e.
$$\phi_i : Q|_{U_i} \stackrel{\cong}{\to} U_i \times P\UU(\HH)$$
and the $K$ action on the right hand side comes form a stable
homomorphism $f_i : K \to P\UU(\HH)$ satisfying the equation
$\phi_i(g \cdot q) = (x, f_i(g)F)$ where $\phi(q) = (x,F)$ and $g
\in K$. Define the transition functions $\rho_{ji} : U_i \cap U_j
\to P\UU(\HH)$ through the equations $$\left( \phi_j|_{U_i \cap
U_j} \circ \left( \phi_i|_{U_i \cap U_j} \right)^{-1} \right)(x,F)
= (x, \rho_{ji}(x)F),$$ and note that we have the compatibility
conditions
\begin{eqnarray} \label{compatibility of functions}
\rho_{kj} \circ \rho_{ji} = \rho_{ki},  \ \ \ \ \ \  \rho_{ij}\circ \rho_{ji}= \rho_{ii}=1, \ \ {\rm{and}} \ \  \nonumber
\rho_{ji}(x)^{-1} f_i(g) \rho_{ij}(x) = f_j(g) \end{eqnarray}
for all $x \in U_i \cap U_j$ and $g \in K$.

Now define the categories $\YY, \QQ$ associated to the open cover \cite{Segal} whose objects
 and morphism are respectively
$${\rm{Obj}}(\YY) := \{(y,i) \in Y \times I | y \in U_i \} \cong \bigsqcup_{i \in I} U_i$$
$${\rm{Mor}}(\YY): = \{(y,i,j) \in Y \times I \times I | y \in U_i \cap U_j \} \cong \bigsqcup_{(i,j) \in I^2} U_i \cap U_i$$
$${\rm{Obj}}(\QQ) := \{(y,i,F) \in Y \times I \times P\UU(\HH)| y \in U_i \} \cong \bigsqcup_{i \in I} U_i \times P\UU(\HH)$$
$${\rm{Mor}}(\QQ): = \{(y,i,j,F) \in Y \times I \times I \times P\UU(\HH) | y \in U_i \cap U_j \} \cong \bigsqcup_{(i,j) \in I^2} U_i \cap U_i\times P\UU(\HH)$$
and whose structural maps for $\YY$ are
$$\source(y,i,j)=(y,i), \ \ \ \target(y,i,j)=(y,j), \ \ \ \inverse(y,i,j)=(y,j,i),$$
$$\composition((y,i,j),(y,j,k))=(y,i,k) \ \ \ {\rm{and}} \ \ \identity(y,i) = (y,i,i),$$
and for $\QQ$ are
$$\source(y,i,j,F)=(y,i,F), \ \ \ \target(y,i,j,F)=(y,j,\rho_{ji}(y)F),$$
$$ \inverse(y,i,j,F)=(y,j,i,\rho_{ji}(y)F), \ \ \identity(y,i,F) = (y,i,i,F)$$
$${\rm{and}} \ \ \ \composition((y,i,j,F),(y,j,k,\rho_{ji}(y)F))=(y,i,k,F). $$

The forgetful functor $\beta:\QQ \to \YY$, $(y,i,j,F) \mapsto (y,i,j)$ induces a map at the
level of the geometric realizations that makes the map
$|\QQ| \stackrel{\beta}{\to} |\YY|$ into a projective unitary bundle.

Now define the functors $\Phi: \QQ \to \DD_K$ and $\phi: \YY \to
\CC_K$  that at the level of morphisms are respectively
$$\Phi(y,i,j,F) := (f_i, F, \rho_{ji}(y)F),  \ \ \ \phi(y,i,j) := (f_i, \rho_{ij}(y))$$
(the equations in (\ref{compatibility of functions}) imply that $\Phi$ and $\phi$ are indeed functors),
which at the level of morphisms make the following diagram commute
$$\xymatrix{
\QQ \ar[d]^\beta \ar[r]^\Phi & \DD_K \ar[d]^\gamma & (y,i,j,F) \ar@{|->}[d]^\beta \ar@{|->}[r]^\Phi & (f_i,F,\rho_{ji}(y)F) \ar@{|->}[d]^\gamma \\
\YY \ar[r]^\phi & \CC_K & (y,i,j) \ar@{|->}[r]^\phi & (f_i,
\rho_{ij}(y)). }$$

Endowing the category $\QQ$ with the left $K$-action
$$K \times \QQ \to \QQ,  \ \ (g,(y,i,j,F)) \mapsto (y,i,j,f_i(g)F)$$
we see that the functor $\Phi$ is $K$-equivariant. Therefore we get the at the level of the geometric realizations
$$\xymatrix{P\UU(\HH) \ar[d] & P\UU(\HH) \ar[d] \\
|\QQ| \ar[d]^{|\beta|} \ar[r]^{|\Phi|} & |\DD_K| \ar[d]^{|\gamma|}   \\
|\YY| \ar[r]^{|\phi|} & |\CC_K| . }$$ we obtain a map of
projective unitary stable $K$-equivariant bundles.

Following the procedure defined in \cite[Section 4]{Segal}, from a partition of unity subordinated to the open cover
$\{U_i\}_{\{i \in I\}}$ one can construct maps $\Theta, \theta$
$$\xymatrix{
Q \ar[r]^{\Theta} \ar[d] & |\QQ| \ar[d]^{|\beta|}   \\
Y \ar[r]^{\theta} & |\YY| .
}$$
such that $\Theta$ is a map of projective unitary stable $K$-equivariant bundles, and moreover
 such that $\Theta$ and $\theta$
are homotopy equivalences.

We have then constructed a map of projective unitary stable $K$-equivariant bundles
$$\xymatrix{
Q \ar[r]^{|\Phi| \circ \Theta} \ar[d] & |\DD_K| \ar[d]^{|\gamma|}   \\
Y \ar[r]^{|\phi| \circ \theta} & |\CC_K| . }$$ which by taking
$\alpha:=|\phi| \circ \theta$ implies the proposition.

\end{proof}

\begin{cor}
Let $Y$ be a trivial $K$-space, then $$\Bun_{st}^K(Y, P\UU(\HH))
=[Y , EP\UU(\HH) \times_{P\UU(\HH)} Hom_{st}(K,P\UU(\HH))].$$
\end{cor}
\begin{proof}

Standard arguments in the classification of principal bundles (see
\cite[Section 3]{Segal}, \cite{Milnor1, Milnor2}), together with
Theorem \ref{theorem classifying map for trivial K spaces} implies
that
$$\Bun_{st}^K(Y, P\UU(\HH)) =[Y , |\CC_K|].$$

Now, using Theorem \ref{theorem classifying map for trivial K
spaces} once again for the bundle in (\ref{universal K bundle}) we
get a map of projective unitary stable $K$-equivariant bundles
$$\xymatrix{
EP\UU(\HH) \times Hom_{st}(K,P\UU(\HH)) \ar[r] \ar[d] & |\DD_K| \ar[d]^{|\gamma|}   \\
EP\UU(\HH) \times_{P\UU(\HH)} Hom_{st}(K,P\UU(\HH)) \ar[r] &
|\CC_K| . }$$ which at the level of the horizontal maps are
homotopy equivalences \cite{Segal}. Therefore we could take the
bundle on the left hand side of the previous diagram as our
universal projective unitary stable $K$-equivariant bundle for
trivial $K$ spaces. This finishes the proof.
\end{proof}

\subsection{Gluing of local objects}

We have seen in Section \ref{subsection system of fix points} that
there is a standard procedure in order to construct a $G$-space
whose $K$-fixed point set has a prescribed homotopy type. For the
purpose of this section we would need to construct a contravariant
functor from the category $\OO_G^P$ of canonical orbits to the
category of topological spaces, which at each orbit type $G/K$
specifies a space which classifies the isomorphism classes of
projective unitary $N(K)$ equivariant bundles over trivial
$K$-spaces endowed with a $N(K)/K$ actions

The first choice that comes to one's mind is to make the
assignment
$$G/K \mapsto EP\UU(\HH) \times_{P\UU(\HH)}
Hom_{st}(K,P\UU(\HH)),$$ but unfortunately this assignment fails
to provide a contravariant functor from $\OO_G^P$ to topological
spaces. One way to fix the problem is to choose spaces with the
same homotopy type for each $G/K$, with the extra property that
the functoriality of the category $\OO_G^P$ is encoded in the
selection. Let us see how this selection can be carried out.

Consider the action groupoid $G \ltimes G/K$ associated to the
left action of $G$ on $G/K$, i.e. the space of objects is $G/K$
and the space of morphisms is $G \times G/K$ with $\source(g,h[K])
= h[K] $ and $\target(g,h[K]) = gh[K]$.

 Note that the functor
$$ K \to G\ltimes G/K, \ \ \ k \mapsto (k, [K])$$
from the category defined by the group $K$ to $G \ltimes G/K$ is
an equivalence of categories (Morita equivalence of groupoids as
in \cite{Moerdijk-Madison}).

One would expect that this equivalence of categories would induce
another equivalence of categories obtained by the restriction
functor, from the category
$$Funct_{st}(G\ltimes G/K, P\UU(\HH))\rtimes Maps(G/K, P\UU(\HH))$$ whose objects are stable functors from
$G\ltimes G/K$ to $P\UU(\HH)$ 
\begin{align*} Funct_{st}(G\ltimes G/K,
P\UU(\HH))  := &\\
 \{F \in Funct(G\ltimes G/K, & P\UU(\HH)) \colon F|_K \in Hom_{st}(K,
P\UU(\HH))\},
\end{align*}
and whose morphisms are defined by natural
transformations, 
to the category $$ Hom_{st}(K, P\UU(\HH)) \rtimes P\UU(\HH).$$
Unfortunately again, this is not true in general for topological
groups (see Section \ref{remark equivalence of categories}), but
it is the case when the group $G$ is discrete (see Proposition
\ref{proposition equivalence of categories} ), let us see why.

\vspace{0.5cm}

Consider the projective unitary $G$-equivariant bundle
$$P \stackrel{p}{\to} G/K$$
and let us suppose that there is a trivialization of $P$ $$\phi :
P \stackrel{\cong}{\to} P\UU(\HH) \times G/K$$ as a unitary
projective bundle. Note that such trivialization always exists in
the case that the group $G$ is discrete.

Associated to the trivialization $\phi$ define the map $\psi: G
\times G/K \to P\UU(\HH)$ by
\begin{eqnarray}\label{functor psi}\psi(g,k[H]):= \pi_1(\phi(gx))
\pi_1(\phi(x))^{-1}\end{eqnarray} where $\pi_1$ is the projection
on the first coordinate, and $x$ is any point in $P$ such that
$p(x)=k[K]$.

\begin{lemma}
The map $\psi$ is a functor from $G\ltimes G/K$ to $P\UU(\HH)$.
\end{lemma}

\begin{proof}
The map $\psi$ is well defined; because if we take $y =xF$ with $F
\in P\UU(\HH)$, since $\phi$ is a map of principal bundles, we have
that $\phi(xF)= \phi(x)F$, and therefore
$$\pi_1(\phi(gy)) \pi_1(\phi(y))^{-1}=\pi_1(\phi(gx))F \left(\pi_1(\phi(x))F
\right)^{-1}=\pi_1(\phi(gx)) \pi_1(\phi(x))^{-1}.$$

By definition we have that $\psi(1,k[H])=1$; and the composition
follows from
\begin{eqnarray*}
\psi(h,gk[K]) \psi(g,k[K]) & = & \pi_1(\phi(hgx))
\pi_1(\phi(gx))^{-1} \pi_1(\phi(gx)) \pi_1(\phi(x))^{-1}\\
&= & \pi_1(\phi(hgx)) \pi_1(\phi(x))^{-1} \\
&=& \psi(hg,k[K]).
\end{eqnarray*}
\end{proof}

If we have two different trivializations $\phi_1,\phi_2: P \cong
P\UU(\HH) \times G/K$, they are related via a gauge transformation
$\sigma: G/K \to P\UU(\HH)$ by the equation
$$ \sigma(p(x))= \pi_1(\phi_1(x))  \pi_1(\phi_2(x))^{-1}.$$

Therefore the associated functors $\psi_1$ and $\psi_2$ are
related by the natural transformation $\sigma$ in the following
way
$$ \sigma(gk[K])^{-1} \psi_1(g,k[K]) \sigma(k[K]) =
\psi_2(g,k[K])$$ which can be written as $\sigma^{-1} \psi_1
\sigma= \psi_2$.

If we denote the category $$\widetilde{\CC}_{G/K}: =
Funct_{st}(G\ltimes G/K, P\UU(\HH))\rtimes Maps(G/K, P\UU(\HH))$$
and recalling the category $\CC_K$ from \ref{section universal for
trivial K spaces} we can conclude

\begin{proposition} \label{proposition equivalence of categories}
Let us suppose that for any local object $P
\to G/K$ there exists a trivialization $P \cong P\UU(\HH) \times
G/K$ as principal bundles, then the restriction functor
$R:\widetilde{\CC}_{G/K} \to \CC_K$ which to any functor $\psi$
assigns the homomorphism $R(\psi)(k):= \psi(k,[K])$, is an
equivalence of categories.
\end{proposition}

\begin{proof}
By definition of $\widetilde{\CC}_{G/K}$, any stable functor
restricts to a stable homomorphism. Let us see that the functor is
essentially surjective.

For $f : K \to P\UU(\HH)$ a stable homomorphism, consider
$$P= P\UU(\HH) \times_f G := P\UU(\HH) \times_K G $$ where $K$ acts on $P\UU(\HH)$ via
the homomorphism $f$. By hypothesis $P$ is trivializable via $\phi
: P \stackrel{\cong}{\to} P\UU(\HH) \times G/K$ and therefore the
associated functor $\psi : G \ltimes G/K \to P\UU(\HH)$ is defined
by the formula (\ref{functor psi}). The restriction $R(\psi)$ is a
homomorphism which is conjugate to $f$, and therefore the functor
$R$ is essentially surjective.

The maps on morphisms
$$Hom_{\widetilde{\CC}_{G/K}} (\psi_1, \psi_2) \to Hom_{{\CC}_{K}} (R(\psi_1),
R(\psi_2))$$ are bijective, as any natural transformation is
defined by its value in $[K]$:
$$\sigma(g[K]) = \psi_1(g,[K]) \sigma([K]) \psi_2(g,[K])^{-1}.$$
Therefore the restriction functor is moreover full and faithful,
hence it gives an equivalence of categories.
\end{proof}

\begin{cor} \label{corollary local objects trivial imply equivalence of categories}
If all local objects $P \to G/K$ trivialize as
$P\UU(\HH)$-principal bundles, then the restriction functor
induces a homotopy equivalence
$$R: |\widetilde{\CC}_{G/K}| \stackrel{\simeq}{\to} |\CC_K|.$$
\end{cor}

\subsubsection{Triviality as principal bundles of the local objects}
 \label{remark equivalence of categories} The
triviality of the local objects as principal $P\UU(\HH)$-bundles
may be characterized in topological terms in the following form.

Let $P \to G/K$ be a projective unitary stable $G$-equivariant
bundle over $G/K$ for $K$ compact. We have seen that
$$P \cong P|_{[K]} \times_K G \cong P\UU(\HH) \times_f G$$
as projective unitary stable $G$-equivariant bundles where $K$
acts on $P\UU(\HH)$ through the stable homomorphism $f: K \to
P\UU(\HH)$.

We have already seen that the homotopy class of the map $Bf: BK
\to BP\UU(\HH)$ determines the isomorphism class of the bundle
defined by $f$, and that the isomorphism class of its associated
bundle over $G/K$ is determined by the homotopy class of a map $F
: EG \times_G G/K \to BP\UU(\HH)$ defined using a homotopy
equivalence $EG \times_G G/K \stackrel{\simeq}{\to} BK$. In this
framework we can see that the bundle $P$ is trivial after
forgetting the $G$ action, if the composition of the maps
$$ G/K \to EG \times_G G/K \to BP\UU(\HH)$$
is homotopy trivial; or in other words that the homomorphism in
third cohomology groups \begin{eqnarray} \label{homomorphism in
third cohomology} H^3_G(G/K, \integer) \to H^3(G/K,
\integer)\end{eqnarray} is trivial.

In the case that the group $G$ is discrete, the orbit types are
also discrete and therefore any bundle over $G/K$ is
trivializable. When the group $G$ is a Lie group there is no
reason to expect that the homomorphism of (\ref{homomorphism in
third cohomology}) is trivial in general. On the contrary,
applying the Eilenberg-Moore spectral sequence to the fibration
$G/K \to BK \to BG$ one gets the isomorphism \cite{Wolf}
$$H^*(G/K, \integer) \cong \Tor_{C^*(BG, \integer)}( \integer,
C^*(BK, \integer)),$$ that in principle would make one guess that
there are several cases on which the homomorphism in
(\ref{homomorphism in third cohomology}) is not trivial.

Take for example the case of $G=SU(3)$ and $K=SO(3)$ whose
quotient space is $X_{-1}=SU(3)/SO(3)$ which is known as the Wu
manifold. In this case $H_2(X_{-1}, \integer) = \integer/2$ (see
\cite{Barden}) which implies that $H^3(X_{-1}, \integer) =
\integer/2$. Applying the Serre spectral sequence to the fibration
$G/K \to EG \times_K G \to BG $ we get that the only non trivial
term of the second page at total degree three is
$$E_2^{0,3}=H^0(BSU(3), H^3(X_{-1}, \integer))= \integer/2,$$ which
moreover survives to the infinity page $E_\infty^{0,3} =
\integer/2$. This fact implies that the homomorphism in
(\ref{homomorphism in third cohomology}) is an isomorphism
$$H^3(BSO(3), \integer) \stackrel{\cong}{\to} H^3(SU(3)/SO(3),
\integer)= \integer/2.$$

From the previous discussion we see that we can improve
Proposition \ref{proposition equivalence of categories} in the
following way.

\begin{proposition} \label{proposition restriction functor R is
equivalence of categories} The restriction functor $R
:\widetilde{\CC}_{G/K} \to \CC_K$ is an equivalence of categories,
if and only if the following homomorphism is trivial
$$H^3_G(G/K, \integer) \to
H^3(G/K,\integer).$$
\end{proposition}

\begin{proof}
We have seen that the local triviality of the local objects is
equivalent to the triviality of the homomorphism, and moreover
that it implies that the functor $R$ is an equivalence of
categories.

Now, if the functor $R$ is essentially surjective, then any stable
homomorphism $f: K \to P\UU(\HH)$ is the restriction of a stable
functor $\psi: G \ltimes G/K \to P\UU(\HH)$ , and therefore the
bundle $P\UU(\HH) \times_f G$ can be trivialized with the
information of the functor $\psi$.
\end{proof}

\begin{cor}
The functor $R :\widetilde{\CC}_{SU(3)/SO(3)} \to \CC_{SO(3)}$ is
not an equivalence.
\end{cor}

Once one has classified all projective unitary stable $G$-equivariant bundles $P \to G/K$ over the orbit types $G/K$,
one may proceed to glue all these bundles via the information provided by their gauge groups. When the bundle
$P \to G/K$ is  trivializable as a $P\UU(\HH)$-bundle, we have shown that the category $\widetilde{\CC}_{G/K}$ contains
all the relevant information to perform the gluing, and since we would like to construct the universal equivariant projective unitary stable 
bundle using the categories $\widetilde{\CC}_{G/K}$, we restric to the case on which $G$ is discrete and the groups $K$ are finite.
The general case on which $G$ is not discrete will not be done in this paper since it would take us away
from the main focus of this work which is the study of the category of functors from $G \ltimes G/K $ to $P\UU(\HH)$. 

\section{Universal equivariant projective unitary stable bundle for proper and discrete actions}
\label{chapter universal bundle}

In this chapter we construct a universal model for the equivariant projective unitary stable bundles for proper
actions of a discrete group $G$.
Hence throughout this chapter the group $G$ is discrete and
$\OO_P^G$ has for objects the sets $G/K$ for all $K$ finite
subgroups of $G$.

 \label{section universal bundle}

\subsection{A model for the universal equivariant projective unitary bundle
for the orbit type $G/K$}  We will proceed in a similar way as in Section
\ref{section universal for trivial K spaces} by constructing a
category $\widetilde{\DD}_{G/K}$ together with a functor
$\widetilde\gamma: \widetilde{\DD}_{G/K} \to \widetilde{\CC}_{G/K}$,
whose geometric realization provides a concrete description of the pullback of the $P\UU(\HH)$-bundle
$|\gamma|:|\DD_K| {\to} |\CC_K|$ along the restriction map $R: |\widetilde{\CC}_{G/K}| \to  |\CC_K|$.

\begin{definition}
Let $\widetilde{\DD}_{G/K}$ be the category whose morphisms
${\rm{Mor}}({\widetilde{\DD}_{G/K}})$ are $$Funct_{st}(G\ltimes G/K,
P\UU(\HH)) \times P\UU(\HH)\times Maps(G/K,
P\UU(\HH)),$$ whose objects are
$$Funct_{st}(G\ltimes
G/K, P\UU(\HH)) \times P\UU(\HH)$$ and whose structural
maps are defined by
\begin{eqnarray*} \source(\psi, F, \sigma)& = &(\psi, F)\\
 \target(\psi, F, \sigma)& =&
(\sigma F^{-1} \psi F \sigma^{-1}, \sigma([K]))\end{eqnarray*}
\begin{eqnarray*}
 \composition((\psi, F, \sigma),(\sigma F^{-1}
\psi F \sigma^{-1}, \sigma([K]), \delta))  =
(\psi, F, & \delta \sigma([K])^{-1} \sigma ).
\end{eqnarray*}
\end{definition}

Let us consider the functor that forgets the $P\UU(\HH)$
coordinate
$$\widetilde{\gamma} : \widetilde{\DD}_{G/K} \to \widetilde{\CC}_{G/K}$$
that at the level of morphisms and objects is defined by
$$\widetilde{\gamma}(\psi, F, \sigma):= (\psi,F \sigma^{-1}), \
 \ \widetilde{\gamma}(\psi, F):= \psi. $$

\begin{proposition} \label{proposition bundle for G/K}
The bundle $$P\UU(\HH) \to |\widetilde{\DD}_{G/K}|
\stackrel{|\widetilde{\gamma}|}{\to} |\widetilde{\CC}_{G/K}|$$ is a
projective unitary stable $N(K)$-equivariant bundle.
\end{proposition}

\begin{proof}
Since $\widetilde{\CC}: G/K \mapsto \widetilde{\CC}_{G/K}$ is a functor from the category
 of canonical orbits $\OO_G^P$ to the category of small categories we have an induced
 left action of ${\rm{Mor}}_{\OO_G^P}(G/K,G/K) = N(K)/K$ on $\widetilde{\CC}_{G/K}$ as well as on
  its geometric realization. On the objects of $\widetilde{\CC}_{G/K}$
  the left $N(K)/K$-action is explicitly given by
\[ n  \cdot \psi= \psi^n,\text{ with } \psi^n(g,k[K]) = \psi(g,kn[K]),\]
while on morphisms it is given by $n \cdot (\psi, \sigma) = (\psi^n, \sigma^n)$,
 where $\sigma^n= \sigma \circ r_n$ and $r_n: G/K \to G/K$ is the $G$-equivariant map with $r_n(g[K])=gn[K]$.
 Now, on the one side we have that
 \begin{align*}
 \left( n \cdot (\sigma^{-1} \psi \sigma) \right) (g,k[K]) &= (\sigma^{-1} \psi \sigma)(g,kn[K])\\
 &= \sigma(gkn[K])^{-1} \psi(g,kn[K]) \sigma(kn[K])
 \end{align*}
 and on the other
  \begin{align*}
 \left( (\sigma^n)^{-1} \psi^n \sigma^n \right) (g,k[K]) &= \sigma^n(gk[K])^{-1} \psi^n(g,k[K]) \sigma^n(k[K])\\
 &= \sigma(gkn[K])^{-1} \psi(g,kn[K]) \sigma(kn[K]);
 \end{align*}
therefore we have that with the previous action,
the group $N(K)/K$ acts on the category $\widetilde{\CC}_{G/K}$.

Let us now endow the category $\widetilde{\DD}_{G/K}$ with a
right $P\UU(\HH)$ action and with a left $N(K)$ action, commuting one
with each other.

The left $N(K)$ action $N(K) \times \widetilde{\DD}_{G/K} \to
\widetilde{\DD}_{G/K}$ at the level of morphisms is defined by
$$n \cdot  (\psi, F, \sigma) := (\psi^n, \psi(n,[K])F,
\sigma^n F^{-1} \psi(n,[K]) F)$$ and at the level of objects by
$$n \cdot  (\psi, F) := (\psi^n, \psi(n,[K]) F);$$ 
note that
\begin{align*}
\target(n \cdot (\psi, F ,\sigma)) = (\sigma^nF^{-1}\psi^nF(\sigma^n)^{-1},\sigma^n([K])F^{-1}\psi(n,[K])F)
\end{align*}
and 
\begin{align*}
n\cdot \target (\psi, F ,\sigma) & = n\cdot (\sigma F^{-1} \psi F \sigma^{-1}, \sigma([K]))\\
&= (\sigma^n F^{-1} \psi^n F (\sigma^n)^{-1}, \sigma(n[K])F^{-1} \psi(n,[K]) F ),
\end{align*}
therefore we have that with the previous action $N(K)$ acts on the category $\widetilde{\DD}_{G/K}$.

The right $P\UU(\HH)$ action $\widetilde{\DD}_{G/K} \times
P\UU(\HH) \to \widetilde{\DD}_{G/K}$ can be defined at the level
of morphisms
$$(\psi, F, \sigma) \cdot L := (\psi, FL,\sigma L)$$
and note that
\begin{align*}
\target((\psi,F,\sigma) \cdot L) &= \target(\psi, FL,\sigma L) \\
&= (\sigma L (FL)^{-1} \psi FL (\sigma L)^{-1}, \sigma([K])L) \\
&= (\sigma F^{-1} \psi F\sigma^{-1}, \sigma([K])L) \\
&= (\target(\psi,F,\sigma)) \cdot L.
\end{align*}
Moreover note that
\begin{align*}
\left( n \cdot (\psi,F, \sigma) \right) \cdot L &= (\psi^n, \psi(n, [K])FL, \sigma^n F^{-1} \psi(n,[K])FL) \\
&=(\psi^n, \psi(n, [K])FL, (\sigma L)^n (FL)^{-1} \psi(n,[K])FL)\\
&= n \cdot \left( ( \psi, F,\sigma) \cdot L \right)
\end{align*}
and therefore we see that the $P\UU(\HH)$ action is indeed
an action, that commutes with the composition of the category, and
that it commutes with the $N(K)$ action.

By definition of the functor $\widetilde{\gamma}$ the induced $P\UU(\HH)$
action on $\widetilde{\CC}_{G/K}$ is trivial, and the
induced $N(K)$ action on $\widetilde{\CC}_{G/K}$ coincides with the $N(K)/K$-action on $\widetilde{\CC}_{G/K} $ referred to at the beginning of the proof.

The stability of the action of the subgroup $K$ on the fibers is
satisfied by construction. Therefore we have that the bundle
$$P\UU(\HH) \to |\widetilde{\DD}_{G/K}| \stackrel{|\widetilde{\gamma}|}{\to}
|\widetilde{\CC}_{G/K}|$$ is a unitary stable $N(K)$-equivariant
bundle.
\end{proof}

\subsection{A model for the universal equivariant projective unitary stable
bundle} The individual spaces $|\widetilde{\DD}_{G/K}|$ assemble like the
restrictions of a projective unitary stable $G$-equivariant bundle
to the fixed point sets of the base. We will glue them together
 accordingly in order to obtain a universal projective unitary stable $G$-equivariant bundle.
  

Let $\widetilde\OO_G^P$ denote the following category. Its objects are copies of $G$,
 one for each canonical proper orbit $G/K$, and we will write $G_{G/H}$ for the copy
 of $G$ which belongs to a proper orbit $G/H$. The set of morphisms from $G_{G/H}$ to $G_{G/K}$ is given by
\[ {\rm Mor}_{\widetilde\OO_G^P}(G_{G/H},G_{G/K}) = \{m \in G\ |\ m^{-1}Hm \subset K\}.\]
A morphism $m\in {\rm Mor}_{\widetilde\OO_G^P}(G/H,G/K)$ should be thought of as a
 $G$-equivariant map from $G$ to $G$ given through right multiplication by $m$, hence
 it covers the induced map $r_m: G/H \to G/K, gH \mapsto gmK$. Consequently the
 composition $n \circ m$
 of two morphisms $m\in {\rm Mor}_{\widetilde\OO_G^P}(G/H,G/K)$ and $n\in {\rm Mor}_{\widetilde\OO_G^P}(G/K,G/L)$
  is defined as the right multiplication by $m \cdot n$.

There is a canonical functor $\pi: \widetilde\OO_G^P \to \OO_G^P$, which maps the object $G_{G/K}$ to $G/K$, and which maps a morphism $m\in {\rm Mor}_{\widetilde\OO_G^P}(G_{G/H},G_{G/K})$ to $r_m: G/H \to G/K$, i.e.~for every morphism $m$ the functor $\pi$ decodes the commutative diagram
$$\xymatrix{G_{G/H} \ar[r]^{m} \ar[d] & G_{G/K} \ar[d]
\\ G/H \ar[r]^{r_m} & G/K.} $$

The categories $\widetilde{\DD}_{G/K}$ provide a contravariant functor $\widetilde{\DD}$ from ${\widetilde\OO_G^P}$ to the category of small categories which on the level of objects is given by
\[ \widetilde{\DD}(G_{G/K}) = \widetilde{\DD}_{G/K},\]
while for a morphism $m\in{\rm Mor}_{\widetilde\OO_G^P}(G_{G/H},G_{G/K})$ the corresponding map
\[ \widetilde{\DD}(m): \widetilde{\DD}_{G/K} \to \widetilde{\DD}_{G/H}\]
is given at the level of morphisms by
\[ \widetilde{\DD}(m)  (\psi, F, \sigma) = (\psi^m, \psi(m,[K])F, \sigma^m F^{-1}\psi(m,[K])F )\]
and at the level of objects by
\[ \widetilde{\DD}(m)  (\psi, F) = (\psi^m, \psi(m,[K])F)\]
where $\psi^m$ and $\sigma^m$ are defined by the formulas
\begin{eqnarray*}
  \psi^m(g,h[H]) & := & \psi(g,hm[K])\\
  \sigma^m(h[H]) & := & \sigma(hm[K]).
\end{eqnarray*}

Note that $\widetilde{\DD}$ respects composition in
$\widetilde\OO_G^P$ since for two composable morphisms $m:G_{G/H}
\to G_{G/K}$ and $n: G_{G/K} \to G_{G/L}$ and a morphism
$(\psi,F,\sigma)$ in $\widetilde\DD_{G/L}$ we have that
\begin{align}
\nonumber \widetilde{\DD}(m) \circ \widetilde{\DD}(n) & (\psi, F, \sigma)\\
\nonumber  = &  \widetilde{\DD}(m) (\psi^n, \psi(n,[L])F, \sigma^n F^{-1}\psi(n,[L])F)\\
\nonumber = &  \left((\psi^n)^m, \psi^n(m,[K])\psi(n,[L])F, \right. \\
\label{explicit calculation} &  \left. (\sigma^n)^m
F^{-1}\psi(n,[L])F
(\psi(n,[L])F)^{-1}\psi^n(m,[K])\psi(n,[L])F\right)\\
\nonumber  = & \left(\psi^{m \cdot n}, \psi(m, n[L]) \psi(n,[L])
F, \sigma^{m\cdot n}F^{-1}\psi(m,n[L])\psi(n,[L])F\right)\\
\nonumber  = & \left(\psi^{m \cdot n}, \psi(m \cdot n,[L]) F,
\sigma^{m\cdot n}F^{-1}\psi(m \cdot n,[L]) F\right)\\
  = & \widetilde{\DD}(n \circ m)
(\psi, F, \sigma). \nonumber
\end{align}
Note also that there is the antihomomorphism
\[{Mor}_{\widetilde\OO_G^P}(G_{G/K},G_{G/K}) \to N(K),\ \ \ \ \ \ n \circ m \mapsto m \cdot n,\]
 and that the contravariant action of ${\rm{Mor}}_{\widetilde\OO_G^P}(G_{G/K},G_{G/K})$ on $\widetilde{\DD}_{G/K}$
  coincides with the left $N(K)$ action introduced in the proof of Proposition \ref{proposition bundle for G/K}.

\begin{rem}
The restriction functors $R':\widetilde\DD_{G/K} \to \DD_K$ and
$R:\widetilde\CC_{G/K} \to \CC_K$ induced by the inclusion $K
\ltimes K/K \hookrightarrow G \ltimes G/K$ are compatible with the
functors $\widetilde\gamma$ and $\gamma$ respectively, so
that we have a commutative diagram
$$\xymatrix{|\widetilde\DD_{G/K}| \ar[r]^{R'} \ar[d] & |\DD_K|\ar[d]
\\ |\widetilde\CC_{G/K}| \ar[r]^{R} & |\CC_K|.} $$
The diagram is a map of $K$-equivariant $P\UU(\HH)$-bundles, hence in
 particular the bundle $|\widetilde\gamma|: |\widetilde\DD_{G/K}| \to |\widetilde\CC_{G/K}|$
  is a pullback of the projective unitary stable $K$-equivariant bundle $|\gamma|: |\DD_K| \to |\CC_K|$.

  Moreover, since the restriction functor $R:\widetilde\CC_{G/K} \to
  \CC_K$ is an equivalence of categories
  (see Proposition \ref{proposition equivalence of categories}), the induced map $R:|\widetilde\CC_{G/K}| \to
  |\CC_K|$ is a homotopy equivalence and therefore the bundle
  $|\widetilde\DD_{G/K}|\to |\widetilde\CC_{G/K}|$ is another
  universal bundle for projective unitary $K$-equivariant bundles
  of $K$-trivial spaces. The main difference between the two
  bundles $|\widetilde\DD_{G/K}|$ and $|\widetilde\DD_K|$ is that
 on the first we have a canonical way to extend the $K$ action  to
 the group $N(K)$, meanwhile on the second we do not know whether
 the action could be extended in a canonical way. It is precisely this
 extension property which later on will allow us to glue together the various
 $|\widetilde\DD_{G/K}|$ to obtain $G$-space.
\end{rem}

The various functors $\widetilde\gamma: \widetilde\DD_{G/K} \to \widetilde\CC_{G/K}$ provide a natural transformation
 $\widetilde\gamma$ from $\widetilde\DD$ to $\widetilde\CC \circ \pi$.
 Now for constructing the universal bundle consider the covariant functors
which at the level of objects are given by
\begin{align*}
\widetilde\nabla & :  \widetilde\OO_G^P \to Spaces,\ G_{G/K} \mapsto G\\
\nabla & :  \OO_G^P \to Spaces, \ G/K \mapsto G/K
\end{align*}
and at the level of morphisms are the $G$-equivariant maps defined
by multiplication on the right.

Assemble the covariant functors $\widetilde{\nabla} $ and $\nabla$
with the contravariant functors $|\widetilde{\DD}|$ and
$|\widetilde{\CC}|$ respectively in order to get spaces that we
denote
\[\EE := |\widetilde{\DD}| \times_{h\widetilde\OO_G^P} \widetilde{\nabla} \ \ \ {\rm{and}}
\ \ \ \  \BB:=|\widetilde{\CC}| \times_{h\OO_G^P} \nabla\ \] where
we have that $|\widetilde{\DD}|(G_{G/K}):=
|\widetilde{\DD}_{G/K}|$ and $|\widetilde{\CC}|({G/K}):=
|\widetilde{\CC}_{G/K}|$.

 The natural transformation
$\widetilde\gamma: \widetilde\DD \to \widetilde\CC \circ \pi$
provides a map
 \[ |\widetilde{\DD}| \times_{h\widetilde\OO_G^P} \widetilde{\nabla}
 \longrightarrow |\widetilde{\CC}| \times_{h\OO_G^P} \nabla\]
 which makes
 \[ P\UU(\HH) \to \EE \to \BB\] into a projective unitary stable $G$-equivariant bundle.

\begin{theorem} \label{theorem the universal bundle}
The bundle
$$P\UU(\HH) \to \EE \to \BB$$
is a universal projective unitary stable $G$-equivariant bundle
for proper $G$-actions.
\end{theorem}
\begin{proof}

The bundles $$P\UU(\HH) \to |\widetilde{\DD}_{G/K}| \times_{N(K)}
G  \to |\widetilde{\CC}_{G/K}| \times_{N(K)/K} G/K$$ are
projective unitary stable $G$-equivariant bundles, and therefore
the bundle $\EE \to \BB$ is also a projective unitary stable
$G$-equivariant bundle since we defined the spaces $\EE$ and $\BB$
via a homotopy colimit.

 Let us check the universality: take
a projective unitary stable $G$-equivariant bundle $P \to X$ over
the proper $G$-space $X$ and take any point $x \in X$. By
Definition \ref{def projective unitary G-equivariant stable
bundle} we know that there exists a $G$-invariant neighborhood $V$ of
$x$ and a $G_x$-contractible slice $U$ such that
$$\alpha: P|_{V} \stackrel{\cong}{\to} (P\UU(\HH) \times U) \times_{G_{x}} G$$
where $G_x$ acts on $P\UU(\HH)$ via a stable homomorphism $f$.
Contracting $U$ to the point $x$, and trivializing the bundle in
the bottom with the map $\phi$ we get the diagram of projective
unitary stable bundles
$$\xymatrix{
P|_{V}\ar[r]^(0.3){\alpha}_(0.3)\cong &  (P\UU(\HH) \times U)
\times_{G_{x}} G
\ar[d]^p &\\
& P\UU(\HH) \times_{f} G \ar[r]^\phi & P\UU(\HH) \times G/G_x. }$$

Define the map
$$P|_{V} \to Obj(\widetilde{\DD}_{G/G_{x}}), \ \ \ z \mapsto (\psi,
\pi_1 (\phi (p (\alpha(z)))))$$ where $\psi : G \ltimes G/K \to
P\UU(\HH)$ is the functor induced by $\phi$ constructed in
\eqref{functor psi}, and note that this map is a $G$-equivariant
map of projective unitary stable bundles.

Changes on the trivialization $\phi$ are parameterized by the
gauge group ${Maps}(G/G_x, P\UU(\HH))$ and they are controlled by
the morphisms of the category $\widetilde{\DD}_{G/G_{x}}$, and
changes on the base point $x$ are parameterized by the morphisms
in $\widetilde{\OO}_G^P$. The composition of a change of base
point together with a change of trivialization induce a path in
the space $\EE$. Therefore the standard argument of the
classification of principal bundles permits us to conclude that by
choosing a {\it{good $G$ cover}} of $X$, we can find a
$G$-equivariant map of projective unitary stable bundles which
makes the diagram into a pullback diagram
$$\xymatrix{P \ar[r] \ar[d] & \EE \ar[d]
\\ X \ar[r] & \BB.} $$

\end{proof}

We have that the space $\BB$ is the base for the universal projective unitary
stable $G$-equivariant bundle and therefore we have that the isomorphism classes of
 projective unitary stable bundles over a proper $G$-space are in 1-1 correspondence with homotopy classes
of maps to the space $\BB$. 

\subsection{The homotopy type of the classifying space of equivariant projective unitary stable 
bundles}

Let us do a first attempt in trying to obtain a homotopy
equivalent space which relates to equivariant cohomology as in
Corollary \ref{corollary equivalence of homotopy type for G}.
Taking the Borel construction of the bundle $\EE \to \BB$ we
obtain a projective unitary bundle
$$P\UU(\HH) \to \EE \times_G EG \to \BB \times_G EG$$
whose classifying map to $BP\UU(\HH)$ induces an adjoint map
$$\Psi:\BB \to Maps(EG, BP\UU(\HH))$$
which is $G$-equivariant. This map restricted to the fixed point
set of a finite group $K$
\begin{align} \label{homotopy equivalence of B with Maps for K finite}
\Psi^K: \BB^K \stackrel{\simeq}{\to} Maps(EG, BP\UU(\HH))^K \cong Maps (BK,BP\UU(\HH))
\end{align}
is a homotopy equivalence; this follows from Corollary
\ref{corollary equivalence of homotopy type for G} and the
homotopy equivalences $$|\widetilde{\CC}_{G/K}| \simeq |\widetilde{\CC}_{G/K}| \times_{N(K)/K} E(N(K)/K) 
\stackrel{\simeq}{\to} \BB^K.$$ 

When the subgroup $K$ is not finite (the orbit type $G/K$ is not
an object in $\OO_G^P$) the fixed point set $\BB^K$ is empty,
meanwhile $Maps(BK,BP\UU(\HH))$ is far from being empty. Therefore
the map $\Psi$ is not in general a $G$-homotopy equivalence, but in
the case that $G$ is finite we have just shown that $\Psi$ induces
a $G$-homotopy equivalence.

\begin{theorem} \label{theorem homotopy equivalence for finite group}
The $G$-equivariant map $$\Psi : \BB \to Maps(EG, BP\UU(\HH)),$$
which is the adjoint of the classifying map of the projective
unitary bundle $\EE \times_G EG \to \BB \times_G EG$, is a
$G$-homotopy equivalence in the case that $G$ is finite.
\end{theorem}

Now let us see the case where $G$ might not be finite.

\begin{definition} \label{definition category MM}
Let $M : \OO_G^P \to Spaces$ be the contravariant functor that is
defined on objects by
\[M(G/K):=Maps(G/K \times_G EG, BP\UU(\HH)) \]
and that to any morphism  $\alpha \in {\rm{Mor}}_{\OO_G^P}(G/K,
G/H)$ the induced morphism $M(\alpha): M(G/H) \to M(G/K)$ is
obtained by the composition with the induced map $G/K \times_G EG
\to G/H \times_G EG$.

Define the space $\MM:= M
\times_{h\OO_G^P} \nabla$.
\end{definition}

By the classification of principal bundles, there exist horizontal
maps making the following diagram commutative
$$\xymatrix{
\left(|\widetilde{\DD}_{G/K}| \times_{N(K)} G\right)\times_G EG \ar[r] \ar[d]  &
EP\UU(\HH)
\ar[d] \\
\left(|\widetilde{\CC}_{G/K}| \times_{N(K)/K} G/K\right) \times_G EG \ar[r] &
BP\UU(\HH).}$$ By taking the adjoint map of the bottom horizontal
map, we get a $N(K)/K$ equivariant map
$$|\widetilde{\CC}_{G/K}| \to Maps(G/K \times_G EG, BP\UU(\HH))$$
for all orbit types; these maps can be assembled into a natural transformation of functors
between the functors $|\widetilde{\CC}|$ and $M$ and therefore we get a $G$-equivariant map
  $$\Psi: \BB \to \MM.$$

\begin{proposition} \label{proposition BB = MM}
The map $\Psi: \BB \to \MM$ is a  $G$-equivariant homotopy
equivalence.
\end{proposition}

\begin{proof}
Let $K$ be a finite subgroup of $G$, and let us consider
restriction of $\Psi$ to the fixed point set of the group $K$. We
obtain the following diagram
$$\xymatrix{ \BB^K \ar[r]^{\Phi^K} & \MM^K \\
|\widetilde{\CC}_{G/K}| \ar[u]_\simeq \ar[d]^\simeq \ar[r] &
Maps(G/K \times_G
EG, BP\UU(\HH)) \ar[u]_\simeq  \ar[d]^\simeq\\
|\CC_K|  \ar[r]^(0.3){\simeq} & Maps(BK,BP\UU(\HH)) }$$ where the
top vertical arrows are homotopy equivalences induced by the
inclusion of the spaces associated to the orbit type $G/K$, the
bottom vertical arrows are homotopy equivalences induced by
restriction (see Corollary \ref{corollary local objects trivial
imply equivalence of categories}), and the bottom horizontal arrow
is a homotopy equivalence as it was shown in Corollary
\ref{corollary equivalence of homotopy type for G}.

Therefore $\Psi$ induces a homotopy equivalence for all finite
subgroups of $G$ which by Theorem 7.4 of \cite{DavisLueck} implies
that $\Psi$ is a $G$-equivariant homotopy equivalence.

\end{proof}

Note that from the construction of the functor $M$ and from
Corollary \ref{corollary equivalence of homotopy type for G} we
see that $\MM$ is a classifying space for the degree 3
cohomology of finite groups of $G$, namely, for all orbit types
$G/K$ for $K$ finite, we have the isomorphism in homotopy groups
$$\pi_i((Maps(G/K,\MM)^G) \cong H^{3-i}(BK,\integer).$$

\subsection{Classification of projective unitary stable equivariant
bundles for proper and discrete actions}
We have seen in \eqref{homotopy equivalence of B with Maps for K finite}  that the map
$\Psi : \BB \to Maps(EG, BP\UU(\HH))$ is a homotopy equivalence once restricted to the
fixed point set of a finite subgroup $K$ of $G$. This fact is precisely what is needed in order to prove the following theorem.
\begin{theorem}
For $X$ a proper $G$-space, the map $\Psi$ induces a bijective map between the
isomorphism classes of projective unitary stable $G$-equivariant
bundles over $X$ and the elements of the third $G$-equivariant cohomology group of
$X$, i.e.
$$\widetilde{\Psi}: \Bun_{st}^G(X,P\UU(\HH)) \stackrel{\cong}{\To} H^3(X \times_G EG; \integer).$$
Therefere, the isomorphism classes of projective unitary stable $G$-equivariant
bundles over $X$ are classified by the elements in $H^3(X \times_G EG, \integer)$.
\end{theorem}
\begin{proof}
Consider the following isomorphisms \begin{eqnarray*} \Bun_{st}^G(X,P\UU(\HH))& \cong &
\pi_0(Maps(X,\BB)^G)\\
&\cong &\pi_0(Maps(X,Maps(EG,BP\UU(\HH)))^G) \\
& \cong & \pi_0(Maps(X\times_G EG, BP\UU(\HH))) \\
& \cong & H^3(X \times_G EG, \integer).
 \end{eqnarray*} where the  isomorphism of the first line follows from Theorem
\ref{theorem the universal bundle}, the
isomorphism of the third line follows from the compact open
topology, and the isomorphism in the fourth line follows from the
fact that $BP\UU(\HH)$ is an Eilenberg-Maclane $K(\integer,3)$
space. We are left with the isomorphism of the second line.

In view of the results of section \ref{subsection system of
fix points}, let us consider the map
$${\rm Hom}_{\OO_G^P}(\Phi X, \Phi \BB) \to {\rm Hom}_{\OO_G^P}(\Phi X, \Phi Maps(EG,BP\UU(\HH))$$
induced by the map $\Psi$. At the orbit type $G/K$ with $K$ finite, we get that the map
$${\rm Maps}(X^K , \BB^K) \stackrel{\simeq}{\to} {\rm Maps}(X^K , Maps(EG,BP\UU(\HH)^K), \ \ f \mapsto \Psi^K \circ f$$
induces a homotopy equivalence since we know that the map $\Psi^K$ defined in
 \eqref{homotopy equivalence of B with Maps for K finite} induces a homotopy equivalence. Therefore we can conclude that
 the map 
 $${\rm Hom}_{\OO_G^P}(\Phi X, \Phi \BB) \stackrel{\simeq}{\To} {\rm Hom}_{\OO_G^P}(\Phi X, \Phi Maps(EG,BP\UU(\HH))$$
 is a homotopy equivalence, and this implies that the map
 $$Maps(X,\BB)^G \stackrel{\simeq}{\To} Maps(X,Maps(EG,BP\UU(\HH)))^G$$
 is also a homotopy equivalence.
This shows the isomorphism of the second  line. Hence we obtain the desired isomorphism 
$$\Bun_{st}^G(X,P\UU(\HH)) \cong H^3(X \times_G EG; \integer).$$
\end{proof}

\section{Twisted equivariant K-theory for proper actions, definition and properties}
\label{chapter Twisted equivariant K-theory}

In this chapter we show that the twisted equivariant K-theory defined
as the homotopy groups of bundles of Fredholm operators satisfies the axioms of 
an equivariant generalized cohomology theory. We first give the definition of the twisted equivariant K-theory for proper actions of Lie groups
twisted by equivariant projective unitary bundles, and then we show that these groups satisfy the axioms of a
generalized equivariant cohomology theory in the sense of
\cite{Lueck1}, namely that this generalized cohomology theory is endowed with induction and restriction structures.

\subsection{Representing space for K-theory}

To define the twisted equivariant K-theory we need to use the
appropriate representing space of K-theory defined by Fredholm
operators, with the extra property that the group $P\UU(\HH)$
endowed with the compactly generated compact open topology acts
continuously on it by conjugation.

The space $\Fred(\HH)$ of Fredholm operators on the Hilbert space
$\HH$ endowed with the norm topology is a representing space for
K-theory, but the group $P\UU(\HH)$ with the compactly generated
compact open topology does not act continuously on $\Fred(\HH)$.
Atiyah and Segal in \cite{AtiyahSegal} construct an alternative
space $\Fred'(\HH)$ representing K-theory by Fredholm operators
and on which $P\UU(\HH)$ acts continuously. Let us recall the
definition of $\Fred'(\HH)$.

\begin{definition}\cite[Chapter 3]{AtiyahSegal} Let $\Fred'(\HH)$
consist of pairs $(A,B)$ of bounded operators on $\HH$ such that
$AB -1$ and $BA -1$ are compact operators. Endow $\Fred'(\HH)$
with the topology induced by the embedding \begin{eqnarray*}
\Fred'(\HH) & \to & {\mathsf{B}}(\HH) \times  {\mathsf{B}}(\HH)
\times {\mathsf{K}}(\HH)
\times {\mathsf{K}}(\HH) \\
(A,B) & \mapsto & (A,B,AB-1, BA-1)
\end{eqnarray*}
where ${\mathsf{B}}(\HH)$ is the bounded operators on $\HH$ with
the compact open topology and ${\mathsf{K}}(\HH)$ is the compact
operators with the norm topology.
\end{definition}

Proposition 3.1 of \cite{AtiyahSegal} shows that $\Fred'(\HH)$ is
a representing space for K-theory and moreover that $\UU(\HH)_{c.o.}$
with the  compact open topology acts
continuously on $\Fred'(\HH)$ by conjugation. It is a simple exercise in topology
to show that the continuity of the $\UU(\HH)_{c.o.}$ action implies the
continuity of the $\UU(\HH)_{c.g.}$ action; therefore we can conclude that the
group $P\UU(\HH)$ acts continuously on $\Fred'(\HH)$ by
conjugation.

Let us choose the identity operator $({\rm{Id}},{\rm{Id}})$ as the
base point in $\Fred'(\HH)$.

\subsection{Definition of Twisted Equivariant K-theory for proper actions}

Let $X$ be a proper $G$ space and $P \to X$ a projective unitary
stable $G$-equivariant bundle over $X$. Recall that the space of
Fredholm operators defined above is endowed with a continuous
right action of the group $P\UU(\HH)$ by conjugation, therefore we
could take the associated bundle over $X$
$$\Fred(P) := P \times_{P\UU(\HH)} \Fred'(\HH)$$
 with fibres $\Fred'(\HH)$ with the induced $G$ action given by
 $$g \cdot [(\lambda, (A,B))] := [(g \lambda, (A,B))]$$
for $g$ in $G$, $\lambda$ in $P$ and $(A,B)$ in $\Fred'(\HH)$.

Denote by $$\Gamma(X; \Fred(P))$$ the space of sections of the
bundle $\Fred(P) \to X$ and choose as base point in this space the
section which chooses the base point on each fiber. This section
exists because the $P\UU(\HH)$ action on $({\rm{Id}},{\rm{Id}})$
is trivial, and therefore $$X \cong P/P\UU(\HH) \cong P
\times_{P\UU(\HH)} \{({\rm{Id}},{\rm{Id}}) \} \subset \Fred(P);$$
let us denote this {\it{identity section}} by $s$.

 The group $G$ acts on
$\Gamma(X; \Fred(P))$ in the natural way, namely for $g \in G$ and
$\sigma$ a section $(g \cdot \sigma)(x) := g \sigma( g^{-1}x)$ and
therefore the fixed point set
$$\Gamma(X; \Fred(P))^G$$
is precisely the space of $G$-equivariant sections. Note that the
base point in $\Gamma(X; \Fred(P))$ is fixed by $G$ because the
identity operators commute with all operators, and therefore the
space of $G$-equivariant sections has also a base point.

\begin{definition} \label{definition K-theory of X,P}
  Let $X$ be a connected $G$-space and $P$ a projective unitary stable
  $G$-equivariant bundle over $X$. The {\it{Twisted $G$-equivariant
  K-theory}} groups of $X$ twisted by $P$ are defined as
  $$K^{-n}_G(X;P) := \pi_n \left( \Gamma(X;\Fred(P))^G, s \right)$$
  where the base point $s$ is the identity section.
\end{definition}

For an inclusion $j: A \to X$ of $G$-spaces we have a restriction
map on invariant sections
$$\Gamma(X; \Fred(P))^G \to \Gamma(A; \Fred(P|_A))^G$$
which induces homomorphisms at the level of the homotopy groups,
hence homomorphisms in the twisted equivariant K-theories
$$j^* : K^*_G(X;P) \to K^*_G(A;P|_A).$$

The relative K-theory groups for the pair $(X,A)$ twisted by the
bundle $P$ over $X$, whenever $X$ is connected, will be defined as
the homotopy groups of the homotopy fiber of the restriction map
$$j^*:\Gamma(X; \Fred(P)) \to \Gamma(A; \Fred(P|_A)).$$

The homotopy fiber of the restriction map can be written in terms
of the mapping cylinder of the inclusion $j:A \to X$
$${\rm{Cyl}}(X,A) = (X \sqcup A \times [0,1]) / a \sim (a,0) \ \  \forall a \in A,$$
as the space of relative sections
$$\Gamma({\rm{Cyl(X,A)}},A\times\{1\};\Fred({\rm{Cyl}}(P,P|_A)))$$
where the relative sections for an inclusion $B \to Y$ are defined
as
$$\Gamma(Y,B;\Fred(Q)) := \{ \sigma \in \Gamma(Y; \Fred(Q)) \colon \sigma|_B = s \}.$$

 \begin{definition} \label{definition relative twisted k-theory groups} The {\it{Relative Twisted Equivariant K-theory
 groups}} of the triple $(X,A;P)$ are the groups
 $$K^{-n}_G(X,A;P) := \pi_n \left(\Gamma({\rm{Cyl(X,A)}},A\times \{1\};\Fred({\rm{Cyl}}(P,P|_A)))^G
 \right).$$

If the space $X$ is a disjoint union of $G$-spaces
$X=\bigsqcup_\alpha
 X_\alpha$, we define the Twisted Equivariant K-theory groups as
 $$K^{-n}_G(X, A;P) := \prod_\alpha K^{-n}_G(X_\alpha, A \cap X_\alpha;P|_{X_\alpha}).$$

\end{definition}

\subsection{Properties of the Twisted Equivariant K-theory groups}

\subsubsection{Functoriality.} Let us consider the category ${\it
prop}\; G-\mathcal{CW}^{2}_{\rm twist}$ whose objects are triples
$(X,A;P)$ consisting of an inclusion $A \to X$ of proper $G$-CW
spaces, together with a projective unitary stable $G$-equivariant
bundle $P \to X$, and whose morphisms $f : (Y,B;Q) \to (X,A; P)$
consist of  $G$ equivariant maps $f : Q \to P$ of principal
$P\UU(\HH)$-bundles, inducing an equivariant map
$\overline{f}:(Y,B) \to (X,A)$ on the bases such that the diagram
$$\xymatrix{ Q \ar[r]^f \ar[d] & P \ar[d] \\
Y \ar[r]^{\overline{f}} & X }$$ is a pullback diagram

Then the map $f$ induces a pullback diagram at the level of the
associated bundles
$$\xymatrix{ \Fred(Q) \ar[r]^f \ar[d] & \Fred(P) \ar[d] \\
Y \ar[r]^{\overline{f}} & X }$$ which implies that any section on
$\Fred(P)$ defines a unique section on $\Fred(Q)$. Thus the map
$f$ induces an equivariant map at the level of sections
$$f^\#: \Gamma(X;\Fred(P)) \to \Gamma(Y;\Fred(Q))$$
and therefore $f$ induces a homomorphism at the level of the
homotopy groups of the $G$-invariant sections
$$f^*: \pi_n \left( \Gamma(X;\Fred(P))^G \right) \to \pi_n \left( \Gamma(Y;\Fred(Q))^G \right)$$
which produces the desired homomorphism in Twisted K-theory groups
$$ f^*: K^{-n}_G(X;P) \to K^{-n}_G(Y; Q).$$

The relative case follows the same principle as we have the
following pullback diagram
$$\xymatrix{ \Fred({\rm{Cyl}}(Q,Q|_B)) \ar[r]^f \ar[d] & \Fred({\rm{Cyl}}(P,P|_A)) \ar[d] \\
{\rm{Cyl}}(Y,B) \ar[r] & {\rm{Cyl}}(X,A)}$$ which induces an
equivariant map in relative sections
$$\Gamma({\rm{Cyl(X,A)}},A;\Fred({\rm{Cyl}}(P,P|_A)))
\to \Gamma({\rm{Cyl(Y,B)}},B;\Fred({\rm{Cyl}}(Q,Q|_B)))$$ inducing
the desired homomorphism in twisted equivariant K-theory groups
$$ f^*: K^{-n}_G(X,A;P) \to K^{-n}_G(Y,B; Q).$$

 This implies that the twisted equivariant K-theory groups provide a functor
from ${\it prop}\; G-\mathcal{CW}^{2}_{\rm twist}$ to graded
abelian groups
$$K_G^{*}: {\it prop}\; G-\mathcal{CW}^{2}_{\rm
twist} \to \mbox{Graded abelian groups}.$$

\subsubsection{Twisted Equivariant K-theory as a generalized cohomology theory.}
The twisted equivariant K-theory groups satisfy the axioms of a
generalized cohomology theory.

\begin{itemize}
\item{\bf{Homotopy axiom:}} Two morphisms $f_0, f_1 :(Y,B;Q) \to
(X,A;
  P)$ in ${\it prop}\; G-\mathcal{CW}^{2}_{\rm
twist}$ are homotopy equivalent if there exists a left
$G$-equivariant and right $P\UU(\HH)$ equivariant homotopy $F: Q
\times I \to P$ such that $F(\_,0) = f_0$ and $F(\_,1) = f_1$, and
moreover that the diagram
$$\xymatrix{ Q\times I \ar[r]^F \ar[d] & P \ar[d] \\
Y \times I \ar[r]^{\overline{F}} & X }$$ is a pullback square.

The fact that the square above is a pullback square implies that
there is an induced map on relative sections
\begin{align*}
\Gamma({\rm{Cyl(X,A)}},A;\Fred({\rm{Cyl}}(P,P|_A)))  \to & \\
 \Gamma({\rm{Cyl(Y,B)}} \times I, &B \times
I; \Fred({\rm{Cyl}}(Q,Q|_B)\times I)) \end{align*}
 whose adjoint can be seen as the desired homotopy
$$\Gamma(X,A; \Fred(P))\times I \stackrel{F^\#}{\to} \Gamma({\rm{Cyl(Y,B)}} , B
; \Fred({\rm{Cyl}}(Q,Q|_B)))$$ between the induced maps $f_0^\#$
and $f_1^{\#}$.

Therefore the homomorphisms $$f_0^*,f_1^*:K^{-n}_G(X,A;P) \to
K^{-n}_G(Y,B;Q)$$ are equal.

\item{\bf{Additivity axiom:}} If the space $X$ is a disjoint union of $G$-spaces
$X=\bigsqcup_\alpha
 X_\alpha$,  we have setup in Definition
\ref{definition relative twisted k-theory groups} that the Twisted Equivariant K-theory groups are
 $$K^{-n}_G(X, A;P) := \prod_\alpha K^{-n}_G(X_\alpha, A \cap X_\alpha;P|_{X_\alpha});$$
this is the additivity axiom.
\item{\bf{Excision axiom:}}
 Let  $Z\subset X$  be  an  open, $G$-invariant  subset  such that the  closure
  of  $Z$  is  contained  in the  interior  of  $A$.  Then  the  restriction
    map induced by $(X-Z, A-Z)\to (X,A)$ induces  a homeomorphism of spaces
    of relative sections
    \begin{align}\Gamma({\rm{Cyl}}(X,A),A \times I;\Fred({\rm{Cyl}}(P,P|_A)))
    \stackrel{\cong}{\to} & \nonumber \\
    \Gamma({\rm{Cyl}}(X-Z,A-Z),(A-Z)  \times
    I; & \Fred({\rm{Cyl}}(P|_{X-Z},P|_{X-A})))
    \label{homeomorphism for relative sections}
    \end{align}
because any section $\sigma$ of the bundle
$\Fred({\rm{Cyl}}(P|_{X-Z},P|_{X-A}))$ which restricts to the base
point in $(A-Z) \times I$,  can be uniquely extended to a section
$\overline{\sigma}$ in $\Fred({\rm{Cyl}}(P,P|_A))$ by defining
$\overline{\sigma}_{A \times I}:=s$ and
$\overline{\sigma}|_{{\rm{Cyl}}(P|_{X-Z},P|_{X-A})}:=\sigma$.

Now, since the inclusion of $A$ into $X$ is a $G$-cofibration
the inclusion of pairs of spaces  $$({\rm{Cyl}}(X,A),A
\times\{1\}) \to ({\rm{Cyl}}(X,A),A \times I)$$ is a $G$-homotopy
equivalence, and therefore the induced map on the spaces of
relative sections is also a $G$-homotopy equivalence,
\begin{align*}\Gamma({\rm{Cyl}}(X,A),A \times I;\Fred({\rm{Cyl}}(P,P|_A)))
    \stackrel{\simeq}{\to} & \\
\Gamma({\rm{Cyl}}   (X,A),A \times \{1\}; &
\Fred({\rm{Cyl}}(P,P|_A))).\end{align*}

The previous argument can also be carried out for the pair $(X-Z,
A-Z)$ yielding a $G$-homotopy equivalence
 \begin{align*}\Gamma({\rm{Cyl}}(X-Z,A-Z),(A-Z)  \times
    I;  \Fred({\rm{Cyl}}(P|_{X-Z},P|_{X-A})))
    \stackrel{\simeq}{\to} & \\
    \Gamma({\rm{Cyl}}(X-Z,A-Z),(A-Z)  \times
    \{1\};  \Fred({\rm{Cyl}}(P|_{X-Z}, P|_{X-A} & ))).\end{align*}
    Therefore, the homotopies outlined above, together with the homeomorphism  of (\ref{homeomorphism for relative
    sections}) implies that there is an isomorphism of relative
    groups
$$K_{G}^{n}(X,A;P) \cong K_{G}^{n}(X-Z, A-Z; P|_{X-Z}).$$

 \item{\bf{Long Exact Sequence axiom for pairs:}} We have defined
the relative twisted equivariant K-theory groups as the homotopy
groups of the relative sections on the mapping cylinder
$$K^{-n}_G(X,A;P) := \pi_n \left(\Gamma({\rm{Cyl(X,A)}},A\times \{1\};\Fred({\rm{Cyl}}(P,P|_A)))^G
 \right).$$

The relative sections of the mapping cylinder is weakly
homotopicaly equivalent to the homotopy fiber of the restriction
map
$$\Gamma(X, \Fred(P)) \to \Gamma(A; \Fred(P|_A)).$$
Therefore  the long exact sequence on homotopy groups induces the
long exact sequence for the twisted equivariant K-theory groups
$$ \to K_G^n(X,A;P) \to K_G^n(X;P) \to K_G^n(A,P|_A) \to K^{n+1}(X,A;P) \to.$$

\end{itemize}

\subsubsection{Bott periodicity.}

The existence of a homotopy equivalence
$$\Fred'(\HH) \stackrel{\simeq}{\to} \Omega^2 \Fred'(\HH)$$
proven  in \cite[Theorem 5.1]{atiyahsingerskew},
yields isomorphisms
$$K^{-n}_G(X,A;P) \stackrel{\cong}{\to} K^{-n-2}_G(X,A;P)$$
which makes the twisted equivariant K-theory groups into a
$\integer$-graded 2-periodic cohomology theory if we define the positive
twisted equivariant K-theory groups by the information on $K^0$
and $K^{-1}$; namely, for $p>0$ we define $$K_G^p(X,A;P) = \left\{
\begin{array}{ccl}
K_G^0(X,A;P) & {\rm {if}} & p \ \mbox{ is even}\\
K_G^{-1}(X,A;P) & {\rm {if}} & p \ \mbox{ is odd.}\\
\end{array}\right.$$

\subsubsection{ Twisted Equivariant K-theory over $G/K$.}
Let $P \to G/K$ be a projective unitary stable $G$-equivariant
bundle over $G/K$ for $K$ finite subgroup of $G$, and recall that
$$P \cong P\UU(\HH) \times_K G$$
as equivariant bundles where $K$ acts on $P\UU(\HH)$ by the stable
homomorphism $f:K \to P\UU(\HH)$.

Therefore we have the index map
$$\pi_0\left( \Gamma(G/K;\Fred(P) \right) = \pi_0 \left( \Fred'(\HH)^K
\right) \stackrel{\cong}{\to} R_{S^{1}}(\widetilde{K})$$ where the
second is obtained by the index map and $R_{S^{1}}(\widetilde{K})$
denotes the Grothendieck group of the semi-group $\widetilde{K}=f^*\UU(\HH)$
of isomorphism classes representations on which $S^1= Ker(\widetilde{K} \to K)$ acts by
multiplication. The index map is an isomorphism. It surjective because the $K$ action
on $P\UU(\HH)$ is stable, and the injectivity follows from the
$G$-equivariant contractibility of $\UU(\HH)$.

We can conclude that the twisted equivariant K-theory groups for
the orbit type $G/K$ twisted by $P \to G/K$ are:
$$K_G^0(G/K;P) \cong R_{S^{1}}(\widetilde{K}), \ \ \ \
K^{-1}_G(G/K;P)=0.$$

\subsubsection{Induction structure.}
The twisted equivariant K-theory groups for proper actions can be
endowed with an {\it{Induction structure}} as it is defined in
Section 1 of \cite{Lueck1}. Let $\alpha: H \to G$ be a group
homomorphism  and $X$ be a $H$-proper space such that
$ker(\alpha)$ acts freely on $X$. Let us denote by $X
\times_\alpha G$ the quotient space $(X \times G)/H$ where the
action is defined by $$H \times (X \times G ) \to X \times G \ \ \
\ \ h(x,l) \mapsto (hx,l\alpha(h)^{-1}),$$ and endow the space $X
\times_\alpha G$ with the left $G$ action defined by $g[x,l]:=
[x,gl]$.  Then we must show that there exists natural graded
isomorphisms
$${\rm{ind}}_\alpha: K^*_H(X,A;P) \stackrel{\cong}{\to} K_G^*(X
\times_\alpha G, A\times_\alpha G ;P \times_\alpha G)$$ which are
functorial with respect to homomorphisms of groups $\beta:G \to K$
on which $ker(\beta)$ acts trivially, that induce isomorphisms at
the level of the long exact sequences of the pairs $(X,A)$ and $(X
\times_\alpha G,A \times_\alpha G)$, and moreover that are
compatible with respect to conjugation; these three conditions
will follow from the following lemmas.

\begin{lemma} \label{lemma induction subgroup}
Let $H \subset G$, $X$ be a proper $H$-space and $P$ a projective
unitary stable $H$-equivariant bundle over $X$. Then spaces of
invariant sections
$$\Gamma(X,\Fred(P))^H \cong \Gamma(X \times_H G, \Fred(P
\times_H G))^G$$ are homeomorphic, and therefore we have an
isomorphism
$${\rm{ind}}_H^G : K^*_H(X,P) \stackrel{\cong}{\to} K_G^*(X
\times_H G, \Fred(P \times_H G)).$$
\end{lemma}
 \begin{proof}
Since the left $H$-action commutes with the right $P\UU(\HH)$
action on $P$ we have that \begin{eqnarray*}\Fred(P)\times_H G & =
& \left(\Fred'(\HH) \times_{P\UU(\HH)} P \right)\times_H G  \\&= &
\Fred'(\HH) \times_{P\UU(\HH)} \left( P \times_H G \right) \\ & =
& \Fred(P \times_H G).
\end{eqnarray*}

The $G$-invariant sections in $\Fred(P) \times_H G$ are determined
uniquely by the $H$ invariant sections of $\Fred(P) \times_H H$
and therefore the restriction map induces a homeomorphism
$$R:\Gamma(X \times_H G, \Fred(P)
\times_H G)^G \stackrel{\cong}{\to} \Gamma(X\times_H H,\Fred(P)
\times_H H)^H.$$

 Now, the canonical map $X \to X \times_H H$, $x \mapsto [(x,1)]$ is an
 $H$-equivariant homomorphism and it induces an
 homeomorphism
 $$ \phi:\Gamma(X, \Fred(P))^H \stackrel{\cong}{\to} \Gamma(X\times_H H,\Fred(P) \times_H
 H)^H.$$

 On the level of homotopy groups the homeomorphism $$ \phi^{-1} \circ R: \Gamma(X, \Fred(P))^H \stackrel{\cong}{\to}
  \Gamma(X \times_H G, \Fred(P)
\times_H G)^G$$ induces the desired isomorphism
$${\rm{ind}}_H^G : K^*_H(X,P) \stackrel{\cong}{\to} K_G^*(X
\times_H G,P \times_H G).$$

 \end{proof}

\begin{lemma} \label{lemma induction normal free action}
Let $X$ be a proper $H$-space together and $P$ a projective
unitary stable $H$-equivariant bundle over $X$. Let $N \subset H$
be a normal subgroup of $H$ acting freely on $X$. Then there is a
canonical homeomorphism between the spaces of invariant sections
$$  \Gamma(X,\Fred(P))^H \cong  \Gamma(X/N , \Fred(P/N))^{H/N}   $$
which induces an isomorphism
$${\rm{inv}}_H^N:     K^*_H(X,P) \stackrel{\cong}{\to} K_{H/N}^*(X/N
,P/N).$$
\end{lemma}

\begin{proof}
We have the following homeomorphisms
\begin{eqnarray*}
\Gamma(X,\Fred(P))^H & = &  \left( \Gamma(X,\Fred(P))^N
\right)^{H/N} \\
& \cong &  \Gamma(X/N,\Fred(P)/N)^{H/N} \\
& \cong &  \Gamma(X/N,\Fred(P/N))^{H/N}
\end{eqnarray*}
where the homeomorphism from the first line to the second follows
from the fact that $N$ acts freely on $X$, and the homeomorphism
from the second line to the third follows from the facts that $N$
acts freely on $P$ and that the $H$ action commutes with the
$P\UU(\HH)$-action.

The composition of the homeomorphisms above induces the desired
isomorphism in the twisted K-theory groups
$${\rm{inv}}_H^N:     K^*_H(X,P) \stackrel{\cong}{\to} K_{H/N}^*(X/N
,P/N).$$

\end{proof}

For a homomorphism $\alpha : H \to G$ such that $N := ker(\alpha)$
acts freely on the $H$-proper space $X$ we have the following
diagram of homomorphisms
$$\xymatrix{ H \ar[rr]^\alpha \ar[rd]^p & & G \\
& H/N \ar@{^{(}->}[ru]^{\overline{\alpha}}}$$ which induce the
following homeomorphisms
$$\xymatrix{ \Gamma(X, \Fred(P))^H  \ar[d]^\cong\\
 \Gamma(X/N,\Fred(P/N))^{H/N} \ar[d]^\cong \\
\Gamma(X/N \times_{H/N} G, \Fred(P/N \times_{H/N} G))^G \ar[d]^\cong \\
\Gamma(X \times_{\alpha} G, \Fred(P \times_{\alpha} G))^G}
$$ where
the first comes from Lemma \ref{lemma induction normal free
action}, the second from Lemma \ref{lemma induction subgroup} and
the third from the canonical $G$-equivariant homeomorphism
$$X \times_\alpha G \stackrel{\cong}{\to} X/N \times_{H/N} G,  \ \
[(x,g)] \mapsto [([x],g)].$$

The compositions of the homeomorphisms described above gives us
the desired isomorphism
$${\rm{ind}}_\alpha : = {\rm{ind}}_{H/N}^G \circ {\rm{inv}}^N_H :
K^*_H(X;P) \stackrel{\cong}{\to} K_G^*(X\times_\alpha G,
P\times_\alpha G).$$

Since the induction structure comes from explicit homeomorphisms
at the level of invariant sections, we claim that the twisted
equivariant K-theory groups possess an induction structure as it
is defined in Section 1 of \cite{Lueck1}. We will not reproduce
the proofs here.

From Lemma \ref{lemma induction normal free action} we also obtain the
following relation between the twisted equivariant K-theory
groups and non-equivariant twisted K-theory groups.

\begin{cor}
Let $X$ be a free $G$ space and let $P$ be a projective unitary
stable $G$-equivariant bundle over $X$. Then there a canonical
isomorphism
$$K_G^*(X;P) \stackrel{\cong}{\to} K^*(X/G;P/G)$$
between the $P$-twisted equivariant K-theory groups of $X$ and the
$P/G$-twisted K-theory groups of $X/G$.
\end{cor}

\subsubsection{Discrete torsion twistings.}
In this last section we would like to describe the relation
between the twisted equivariant K-theory groups defined in this
paper and the twisted equivariant K-theory groups defined via
projective representations characterized by discrete torsion as it
is defined in \cite{Dwyer}.

For any cohomology class $\beta \in H^3(BG, \integer)$ denote by
$F_\beta: EG \to BP\UU(\HH)$ a $G$-invariant map whose homotopy
class represents $\beta$.

If $X$ is a proper $G$-space we consider the map
$$X^K \to Maps(EG \times G/K, BP\UU(\HH))^G,  \ \ x \mapsto
F_\beta \circ \pi_1$$ at the level of fixed point sets for all $K$
 compact subgroups of $G$.

 This map can be assembled into a $G$-equivariant map
$$ \Psi_\beta: X \to \MM$$
where $\MM$ is the space defined in
Definition \ref{definition category MM}. By Proposition
\ref{proposition BB = MM} we know that the map $\Psi_\beta$
induces a map $\overline{\Psi}_\beta: X \to \BB$ and therefore
$$(\overline{\Psi}_\beta)^* \EE \to X$$
is a projective unitary $G$-equivariant stable bundle.

We claim that the $\beta$-twisted $G$-equivariant K-theory groups
of $X$ defined in \cite{Dwyer} are isomorphic to the twisted
$G$-equivariant K-theory groups
$$K^*_G(X,(\overline{\Psi}_\beta)^* \EE).$$

The proof will be postponed to a forthcoming publication \cite{BarcenasEspinozaUribeVelasquez} where the
appropriate tools for proving such fact, and many others,  will be
developed.

\bibliographystyle{abbrv}
\bibliography{TwistedK}

\end{document}